\documentclass[11pt,a4paper]{amsart}
\usepackage[utf8]{inputenc}
\usepackage[T1]{fontenc}
\usepackage{amsmath,amssymb,amsthm,mathtools}
\usepackage[margin=2.6cm]{geometry}
\usepackage{enumitem}
\usepackage{graphicx}
\usepackage{caption}
\usepackage[colorlinks=true,linkcolor=blue,citecolor=blue,urlcolor=blue]{hyperref}
\usepackage{xcolor}
\usepackage{tikz}
\definecolor{darkgreen}{rgb}{0,0.5,0}

\newtheorem{theorem}{Theorem}[section]
\newtheorem{proposition}[theorem]{Proposition}
\newtheorem{lemma}[theorem]{Lemma}
\newtheorem{corollary}[theorem]{Corollary}
\theoremstyle{definition}
\newtheorem{definition}[theorem]{Definition}
\newtheorem{example}[theorem]{Example}
\theoremstyle{remark}
\newtheorem{remark}[theorem]{Remark}
\numberwithin{equation}{section}

\newcommand{\R}{\mathbb{R}}
\newcommand{\N}{\mathbb{N}}
\newcommand{\K}{\mathcal{K}}
\newcommand{\G}{\mathcal{G}}

\newcommand{\eps}{\varepsilon}
\newcommand{\dd}{\,\mathrm{d}}
\newcommand{\Om}{\Omega}
\DeclareMathOperator{\supp}{supp}

\begin{document}

\title[The H\'enon equation beyond the critical Sobolev exponent]{The H\'enon equation beyond the critical Sobolev exponent}

\author[Gladiali]{Francesca Gladiali}
\address{Dipartimento di Scienze CFMN,  Universit\`a degli Studi di Sassari, Via Vienna 2 - 07100 Sassari, Italy, e-mail: {\sf fgladiali@uniss.it}.}
\author[Grossi]{Massimo Grossi }
\address{Dipartimento di Scienze di Base e Applicate per l'Ingegneria, Universit\`a di Roma ``La Sapienza'', Via Scarpa 12 - 00161 Roma, Italy, e-mail: {\sf massimo.grossi@uniroma1.it}.}


\begin{abstract}
We consider the Dirichlet problem $-\Delta u=|x|^\alpha u^p$, $u>0$ in $\Omega$, $u=0$ on $\partial\Omega$, where $\Omega\subset\R^N$, $N\ge 3$, is a bounded domain containing the origin and $\alpha>0$. We prove that a solution exists for every exponent
$1<p<p_\alpha=\frac{N+2+2\alpha}{N-2}$, in particular in the whole Sobolev-supercritical range $\frac{N+2}{N-2}\le p<p_\alpha$, provided $\Omega$ is invariant under $O(N-1)\times O(1)$ and satisfies an angular monotonicity condition. This class contains the balls, the oblate ellipsoids of revolution, lens-shaped domains, and non-convex, non-smooth and non star-shaped domains. The solution is obtained by minimizing the Rayleigh quotient on a suitable closed convex cone.
\end{abstract}

\maketitle

%=====================================================================
\section{Introduction}\label{sec:intro}
%=====================================================================

Let $\Omega\subset\R^N$, $N\ge3$, be a bounded domain with $0\in\Omega$, and let $\alpha>0$. We study the H\'enon problem
\begin{equation}\label{eq:P}
\begin{cases}
-\Delta u=|x|^\alpha u^p & \text{in }\Omega,\\
u>0 & \text{in }\Omega,\\
u=0 & \text{on }\partial\Omega,
\end{cases}
\end{equation}
introduced in \cite{He} in the study of rotating stellar systems. Two exponents play a role:
\[
2^*-1=\frac{N+2}{N-2}\qquad\text{and}\qquad p_\alpha:=\frac{N+2+2\alpha}{N-2}>\frac{N+2}{N-2}.
\]
We also write
\[
2^*=\frac{2N}{N-2},\qquad q_\alpha:=p_\alpha+1=\frac{2(N+\alpha)}{N-2}.
\]

If $1<p<2^*-1$, a solution of \eqref{eq:P} exists in every bounded domain $\Omega$ by the mountain pass theorem. When $\Omega=B$ is a ball centred at the origin, Ni \cite{Ni} observed that the space $H^1_{0,\mathrm{rad}}(B)$ of radial functions embeds compactly into $L^{p+1}(B,|x|^\alpha\dd x)$ exactly when $p+1<q_\alpha$. Hence \eqref{eq:P} has a radial solution for every $1<p<p_\alpha$, and in particular for Sobolev-supercritical exponents. On the other hand, the Pohozaev identity \cite{Po} shows that \eqref{eq:P} has no solution if $p\ge p_\alpha$ and $\Omega$ is smooth and star-shaped with respect to the origin. So $p_\alpha$ is the natural critical exponent of the problem.

For non-radial domains and $2^*-1\le p<p_\alpha$ very little is known. The energy
$\frac12\int|\nabla u|^2-\frac1{p+1}\int|x|^\alpha|u|^{p+1}$ is not even defined on $H^1_0(\Omega)$, and without symmetry nothing prevents concentration at points $x_0\ne0$, where the weight $|x|^\alpha$ is harmless and the problem behaves like $-\Delta u=|x_0|^\alpha u^p$. 
Existence results are available for $p=p_\alpha-\eps$ with $\eps>0$ small, in every smooth bounded domain containing the origin, via a perturbative construction of solutions concentrating at the origin: this was done in \cite{GG} for $0<\alpha\le1$, and extended in \cite{GGN} to every $\alpha>0$ with $\alpha\notin2\N$. 
In the unit ball, non-radial solutions were found for $N\ge4$ and large $\alpha$ in \cite{BS}. For even $N\ge4$ and every $\alpha>0$, \cite{CM} obtained positive solutions in a class of functions with a monotonicity property in an angular variable, and proved that they are non-radial for $p$ in a range close to $p_\alpha$; see also \cite{SSW} for the subcritical case.  
In the ball, positive non-radial solutions bifurcating from the radial one were obtained in \cite{AG1} for $0<\alpha\le1$, and in \cite{A} for every $\alpha>0$, using the Morse index of radial solutions for $p$ close to $1$ and for $p$ close to $p_\alpha$, the latter computed in \cite{AG}.

The purpose of this paper is to prove existence of solutions of \eqref{eq:P} in the full range $1<p<p_\alpha$, for every $\alpha>0$ and $N\ge3$, in a class of non-radial domains defined by a symmetry and a monotonicity condition. We write $x=(x',x_N)\in\R^{N-1}\times\R$ and let
\begin{equation}\label{G}
\G:=O(N-1)\times O(1)
\end{equation}
act on $\R^N$ by $(g,t)\cdot(x',x_N)=(gx',tx_N)$. 

\begin{definition}\label{def:adm}
A bounded domain $\Omega\subset\R^N$ is \emph{admissible} if $0\in\Omega$ and
\begin{enumerate}[label=\textup{(\alph*)}]
\item $\Omega$ is $\G$-invariant, i.e.\ axially symmetric around the $x_N$-axis and symmetric with respect to the hyperplane $\{x_N=0\}$;
\item for every $x\in\Omega$ and every $y\in\R^N$ with $|y|=|x|$ and $|y_N|\le|x_N|$, we have $y\in\Omega$.
\end{enumerate}
\end{definition}

Condition (b) says that,  moving a point of $\Omega$ along a meridian of $\partial B_{|x|}$ towards the equator never leaves $\Omega$. For instance, a domain whose boundary is the radial graph $r=g(\theta)$ is admissible if and only if $g$ is non-increasing in the latitude $\theta\in[0,\pi/2]$ (Lemma~\ref{lem:polar}). Balls centred at $0$, oblate ellipsoids of revolution $\{|x'|^2/a^2+x_N^2/b^2<1\}$ with $b\le a$, lenses $B_R(ce_N)\cap B_R(-ce_N)$, and the non-convex, non star-shaped examples of Section~\ref{sec:examples} are admissible.  Prolate ellipsoids and cylinders are not. Some examples are shown in  Figure~\ref{fig:admissible} and Figure~\ref{fig:3d}. 
\vskip0.5cm
To state our results, we need to introduce an appropriate cone $\K\subset H^1_0(\Omega)$, which will play a fundamental role throughout the paper.
For $\varphi\in(0,\pi/2)$ we set
\[
\nu_\varphi:=-\sin\varphi\, e_1+\cos\varphi\, e_N,\qquad \tau_\varphi:=\cos\varphi\, e_1+\sin\varphi\, e_N .
\]
$\{\nu_\varphi,\tau_\varphi\}$ is an orthonormal basis of the plane $\Pi:=\mathrm{span}\{e_1,e_N\}$. Let $\sigma_\varphi$ be the reflection through the hyperplane $\nu_\varphi^\perp$:
\[
\sigma_\varphi x:=x-2(x\cdot\nu_\varphi)\nu_\varphi,
\]
and \[
W_\varphi:=\{x:\ x\cdot\nu_\varphi>0,\ x\cdot\tau_\varphi>0\}.
\]
\begin{definition}\label{def:K}
Let $\Omega$ be admissible. We denote by $\K=\K(\Omega)$ the set of functions $u\in H^1_{0,\G}(\Omega)$ such that $u\ge0$ a.e.\ and, for every $\varphi\in(0,\pi/2)$,
\begin{equation}\label{eq:reflK}
u(\sigma_\varphi x)\ge u(x)\qquad\text{for a.e. }x\in W_\varphi .
\end{equation}
\end{definition} 
The same cone $\K$ was used in annuli in \cite{BCNW} where it was defined by means of a condition on the angular derivative.  The novelty here is its use in  domains containing the origin.

Our main result is the following.

\begin{theorem}\label{thm:main}
Let $N\ge3$, $\alpha>0$, $1<p<p_\alpha$, and let $\Omega\subset\R^N$ be an admissible domain with Lipschitz continuous boundary. Then problem \eqref{eq:P} has a solution $u\in H^1_0(\Omega)\cap C^2(\Omega)\cap C(\overline\Omega)$ with the following properties.
\begin{enumerate}[label=\textup{(\roman*)}]
\item $u$ is $\G$-invariant: $u(x)=\tilde u(|x'|,|x_N|)$.
\item $u$ is monotone in the latitude: for every unit vector $e\in\R^{N-1}\times\{0\}$, every $r>0$ and every $0\le\theta'<\theta\le\pi/2$ such that $x_\theta:=r(\cos\theta\, e+\sin\theta\, e_N)\in\Omega$, we have $x_{\theta'}\in\Omega$ and
\[
u(x_{\theta'})\ge u(x_\theta).
\]
In particular, on each sphere $\partial B_r$ the maximum of $u$ on $\partial B_r\cap\Omega$ is attained on the equator $\{x_N=0\}$.
\item $u$ minimizes the Rayleigh quotient
\[
\frac{\int_\Omega|\nabla v|^2\dd x}{\left(\int_\Omega|x|^\alpha|v|^{p+1}\dd x\right)^{2/(p+1)}}
\]
on the closed convex cone $\K$.
\end{enumerate}
\end{theorem}
In the previous theorem no asymptotic condition is imposed on $\alpha$, on $p_\alpha-p$ or on the shape of $\Omega$.
\vskip0.2cm
\begin{remark}
The Lipschitz regularity of $\partial\Omega$ required in the theorem above is not needed for the existence part. If $\Omega$ is not Lipschitz one can repeat the proof obtaining a solution $u\in H^1_0(\Omega)\cap L^\infty(\Omega)\cap C^2(\Omega)$.  Regularity of $\partial\Omega$ is used only to obtain $u\in C(\overline\Omega)$ with $u=0$ on $\partial\Omega$ (although, to ensure that $u\in C(\overline{\Om})$, it is sufficient to assume that $\partial\Om$ satisfies an exterior cone condition, see \cite{GT}).
Moreover, since \eqref{eq:P} is supercritical, global regularity of weak solutions is not automatic. We prove it by a Brezis--Kato--Moser iteration \cite{BK} performed inside $\K$, using that $\K$ is stable under composition with non-decreasing Lipschitz functions (Proposition~\ref{prop:reg}).
\end{remark}

\begin{remark}
The passage from a minimizer on $\K$ to a solution of \eqref{eq:P} follows  some ideas used in \cite{Mo1,Mo2},
(see also \cite{CM}). 
Related constraints of monotonicity type were introduced in \cite{ST} for supercritical Neumann problems.
\end{remark}

\begin{remark}
\vskip0.1cm
\noindent Admissible domains of the form $\{|x|<g(\theta)\}$ are star-shaped with respect to $0$. When they are smooth (e.g.\ ellipsoids), the Pohozaev identity shows that the range $p<p_\alpha$ is optimal. Admissible domains need not be star-shaped (Example~\ref{ex:nonstar}); for them the existence question for $p\ge p_\alpha$ is open.
\end{remark}

%The method does not extend to nodal solutions in the supercritical range. In the natural convex cone of $\G$-invariant functions which are non-increasing in the latitude, the negative part is non-decreasing in the latitude, so it may concentrate at the poles, where $O(N-1)$-invariance does not prevent blow-up. Hence compactness, and thus existence, holds only for $1<p<\frac{N+2}{N-2}$.

In the ball, Theorem~\ref{thm:main} gives new non-radial solutions in the supercritical range, at least for $\alpha>2$:

\begin{corollary}\label{cor:ball}
Let $\Omega=B$ be the unit ball, $N\ge3$ and $\alpha>2$. There exists $\bar p\in(1,p_\alpha)$ such that for every $p\in[\bar p,p_\alpha)$ the solution given by Theorem~\ref{thm:main} is non-radial.
\end{corollary}

The proof of this corollary (Section~\ref{sec:proof}) combines the computation of the Morse index of the positive radial solution $u_p$ for $p$ close to $p_\alpha$ in \cite[Theorem~1.1]{AG} with the observation that, when $\alpha>2$, $u_p$ has a negative direction of the form $w(|x|)\big(N\frac{x_N^2}{|x|^2}-1\big)$ (for some radial function $w$), which is compatible with the constraint $\K$.  The solutions in Corollary \ref{cor:ball} are different from the non-radial solutions obtained in \cite{BS}, which are invariant under $O(k)\times O(N-k)$ with $2\le k\le N-2$, but are also different from the $O(N-1)$-invariant solutions bifurcating from the radial one found in \cite{AG1} and the ones in \cite{A}.

For $1<p<\frac{N+2}{N-2}$ and $\alpha$ large, the solution of Theorem~\ref{thm:main} in the ball is not a least energy solution. Indeed, least energy solutions are foliated Schwarz symmetric \cite{SW} and non-radial \cite{SSW}, while a foliated Schwarz symmetric function in $\K$ is radial, being invariant under $\G$ and under the rotations around some axis. 
It would be interesting to understand what happens to the solutions of Corollary \ref{cor:ball} as $p\to p_\alpha$.

We now give an outline of the proof of Theorem \ref{thm:main}. The first step is to establish an inequality in the cone $\K$ analogous to (10) in Ni's paper \cite{Ni}. However, since we are not working in a radial setting, several additional difficulties arise. By exploiting the invariance of $\K$ under the orthogonal group $O(N-1)$ and the monotonicity property \eqref{eq:reflK} in the definition of $\K$, we can prove the following result:
\begin{theorem}\label{cor:emb}
Let $\Omega$ be admissible, $\alpha\ge0$, and $1\le q\le q_\alpha$. There is $C=C(N,\alpha,q,\Omega)$ such that
\[
\int_\Omega|x|^\alpha u^q\dd x\le C\|\nabla u\|_2^q\qquad\text{for every }u\in\K .
\]
\end{theorem}
We believe that this inequality is interesting in its own right. Moreover, it immediately yields the compactness of sequences $u_n$ in $\K$. At this point, considering the functional
\[
F(u):=\int_\Omega|x|^\alpha|u|^{p+1}\dd x,\qquad \Sigma:=\{u\in\K:\ F(u)=1\}, \qquad \mathsf S:=\inf_{u\in\Sigma}\int_\Omega|\nabla u|^2\dd x,
\]
it is not difficult to prove the existence of a minimizer. Clearly, this minimizer satisfies a variational inequality, and the final step consists in showing that it also satisfies equation \eqref{eq:P}. This last step is far from trivial as well, and relies on the invariance of $\K$ under the solution operator of the Poisson problem.

Via the Kelvin transform, Theorem~\ref{thm:main} also yields solutions of supercritical problems in exterior domains $\R^N\setminus D$, where $D$ belongs to a class of axially symmetric obstacles dual to admissible domains (Section~\ref{sec:kelvin}).  See Figure \ref{fig:duality}.
\begin{theorem}\label{thm:ext}
Let $N\ge3$, $\beta\in\R$, and
\[
p>\max\Big\{p_\beta,\ \frac{N+2+\beta}{N-2}\Big\},\qquad p_\beta=\frac{N+2+2\beta}{N-2}.
\]
Let $D$ be an admissible obstacle as in Definition \ref{def:polar},  with Lipschitz continuous boundary.  Then the problem
\begin{equation}\label{eq:ext}
-\Delta v=|y|^\beta v^p,\quad v>0\quad\text{in }\R^N\setminus D,\qquad v=0\ \text{on }\partial D,\qquad v(y)\to0\ \text{as }|y|\to\infty,
\end{equation}
has a solution $v\in C^2(\R^N\setminus D)$ with the following properties.
\begin{enumerate}[label=\textup{(\roman*)}]
\item Fast decay: $|y|^{N-2}v(y)\to\ell>0$ as $|y|\to\infty$.
\item $\int_{\R^N\setminus D}|y|^\beta v^{p+1}\dd y<\infty$ and $\int_{\R^N\setminus D}|\nabla v|^2\dd y<\infty$.
\item $v$ is $\G$-invariant and, on each sphere $\partial B_r$, non-increasing from the equator to the poles (in the sense of Theorem~\ref{thm:main}(ii)).
\end{enumerate}
\end{theorem}
In particular, in the Lane--Emden case we obtain:

\begin{corollary}\label{cor:ext}
Let $N\ge3$ and let $D$ be an admissible obstacle as in Definition \ref{def:polar}, with Lipschitz boundary. For every $p>\frac{N+2}{N-2}$ the problem $-\Delta v=v^p$, $v>0$ in $\R^N\setminus D$, $v=0$ on $\partial D$, has a solution with fast decay $|y|^{N-2}v(y)\to\ell>0$ as $|y|\to\infty$.
\end{corollary}

For the exterior of a ball, the existence of a unique positive radial solution with fast decay is classical for every $p>\frac{N+2}{N-2}$, see, for instance, \cite{GP} and \cite{DdPMW}. For a general smooth bounded domain, Dávila, del Pino, Musso and Wei \cite{DdPMW} proved the existence of a fast-decay solution when $p$ is sufficiently close to the Sobolev critical exponent. Later, Cowan in \cite{C} obtained fast-decay solutions for small perturbations of the exterior of a ball for every supercritical exponent outside a discrete set of resonant values. To the best of our knowledge, Corollary~\ref{cor:ext} is the first result providing a fast-decay solution for every $p>\frac{N+2}{N-2}$ on a non-perturbative class of non-spherical obstacles. In particular, the result applies to every prolate ellipsoid of revolution.

\smallskip
A crucial role in our argument is played by the invariance of the cone $\K$ with respect to the symmetry group $O(N-1)\times O(1)$.
It could be replaced  by $O(m)\times O(n)$, $m+n=N$, with reflections pushing points towards $\{x''=0\}$, $x=(x',x'')\in\R^m\times\R^n$. The polarization argument of Section~\ref{sec:inv} carries over. However, when $n\ge2$ functions of the corresponding cone may concentrate on the $(m-1)$-spheres in $\{x''=0\}$, so compactness holds only for $p+1<\frac{2(n+1)}{n-1}$, which covers the whole range $p<p_\alpha$ only if $\alpha<\frac{2(m-1)}{n-1}$.

\smallskip
We conclude this introduction discussing the location of the maximum point. Theorem~\ref{thm:main} excludes concentration away from the origin at the level of the compactness argument. However, it does not say that $u$ attains its maximum at $0$, nor that $u$ is close to Ni's radial solution. 
When $\Omega=B$, both $u$ and the radial solution belong to $\K$, and $\mathsf S\le\mathsf S_{\rm rad}$; by Corollary~\ref{cor:ball} the inequality is strict for $\alpha>2$ and $p$ close to $p_\alpha$. For $0<\alpha\le2$ and $p$ close to $p_\alpha$ the radial solution has Morse index $N+1$ \cite{AG}, but the $N$ negative directions of degree one are not $\G$-invariant, so they are not compatible with $\K$. Whether $\mathsf S=\mathsf S_{\rm rad}$ in this case is open.

\smallskip
The paper is organized as follows. Section~\ref{sec:K} introduces the reflections, the admissible domains and the cone $\K$. Section~\ref{sec:emb} contains the weighted inequalities and the compactness of $\K$. Section~\ref{sec:inv} proves the invariance of $\K$. Section~\ref{sec:proof} proves Theorem~\ref{thm:main}. Sections~\ref{sec:examples} and~\ref{sec:kelvin} are devoted to examples of
admissible domains. We stress that these sections are an
essential part of the paper: they show that the class of admissible domains is
rich; it contains non-convex, non-smooth, non-star-shaped and non-simply
connected domains, as well as exterior domains via the Kelvin transform, 
(see Figures~\ref{fig:admissible}, \ref{fig:3d} and~\ref{fig:duality})
 and
they provide criteria that are easy to check in concrete cases. 

%=====================================================================
\section{Reflections, admissible domains and the cone \texorpdfstring{$\K$}{K}}\label{sec:K}
%=====================================================================

We write $x=(x_1,\dots,x_N)=(x',x_N)$ and $e_1,\dots,e_N$ for the canonical basis. For $x\ne0$ the \emph{latitude} of $x$ is
\[
\theta(x):=\arcsin\frac{|x_N|}{|x|}\in\Big[0,\frac\pi2\Big],\qquad\text{so that}\quad |x'|=|x|\cos\theta(x),\quad |x_N|=|x|\sin\theta(x).
\]
$B_r$ is the open ball of radius $r$ centred at $0$, and $\omega_{N-2}:=\mathcal H^{N-2}(S^{N-2})$ is the measure of the unit sphere of $\R^{N-1}$. We use the following elements of $\G$ (see \eqref{G}): 
\[
A(x):=(-x_1,x_2,\dots,x_N),\qquad B(x):=(x_1,\dots,x_{N-1},-x_N),\qquad C:=AB.
\]
Every $u\in H^1_0(\Omega)$ is extended by zero outside $\Omega$, so that $u\in H^1(\R^N)$. 
We use the norm $\|u\|:=\|\nabla u\|_{L^2(\Omega)}$ on $H^1_0(\Omega)$ and we denote by $H^1_{0,\G}(\Omega)$ the closed subspace of $\G$-invariant functions of $H^1_0(\Omega)$. 

For $\varphi\in(0,\pi/2)$ and using the notations of the introduction, define the open ``quadrants''
\[
Q^+_\varphi:=\{x:\ (x\cdot\nu_\varphi)(x\cdot\tau_\varphi)>0\},\qquad
Q^-_\varphi:=\{x:\ (x\cdot\nu_\varphi)(x\cdot\tau_\varphi)<0\}.
\]
Their complement $Z_\varphi:=\nu_\varphi^\perp\cup\tau_\varphi^\perp$ has measure zero. When no confusion arises we drop the index $\varphi$.

In the plane $\Pi:=\mathrm{span}\{e_1,e_N\}$, with polar angle measured from $e_1$, the point $r(\cos\vartheta\,e_1+\sin\vartheta\,e_N)$ is mapped by $\sigma_\varphi$ to $r(\cos(2\varphi-\vartheta)e_1+\sin(2\varphi-\vartheta)e_N)$. So $\sigma_\varphi$ is the reflection across the line $\R\tau_\varphi$, which makes an angle $\varphi$ with the equator. Points of $W_\varphi$ lie ``above'' this line and are sent ``below'' it, i.e.\ towards the equator. The next lemma makes this precise.

\begin{lemma}\label{lem:geom}
Let $\varphi\in(0,\pi/2)$ and write $\sigma=\sigma_\varphi$, $\nu=\nu_\varphi$, $\tau=\tau_\varphi$.
\begin{enumerate}[label=\textup{(\roman*)}]
\item $\sigma$ is an orthogonal involution with $|\sigma x|=|x|$. %It changes the sign of $x\cdot\nu$ and leaves $x\cdot\tau$ and $x_2,\dots,x_{N-1}$ unchanged {\Ma io questo non lo scriverei, segue automaticamente dalla definizione di $\sigma$}. 
It commutes with $C$, and
\[
\sigma(W_\varphi)=\{x\cdot\nu<0,\ x\cdot\tau>0\},\quad C(W_\varphi)=\{x\cdot\nu<0,\ x\cdot\tau<0\},\quad \sigma(Q^\pm_\varphi)=Q^\mp_\varphi,\quad C(Q^\pm_\varphi)=Q^\pm_\varphi .
\]
Moreover $Q^+_\varphi=W_\varphi\cup C(W_\varphi)$.
\item For every $x$,
\begin{equation}\label{eq:coords}
x_N=\cos\varphi\,(x\cdot\nu)+\sin\varphi\,(x\cdot\tau),\qquad (\sigma x)_N=-\cos\varphi\,(x\cdot\nu)+\sin\varphi\,(x\cdot\tau).
\end{equation}
\item If $x\in\overline{W_\varphi}$, then $x_N\ge0$ and $|(\sigma x)_N|\le |x_N|$. Consequently $|(\sigma x)'|\ge|x'|$ and $\theta(\sigma x)\le\theta(x)$ for $x\ne0$.
\item If $x\cdot\tau=0$, then $(\sigma x)_N=-x_N$ and $|(\sigma x)'|=|x'|$. If $x\cdot\nu=0$, then $\sigma x=x$.
\end{enumerate}
\end{lemma}

\begin{proof}
(i) $\sigma$ is the reflection through a hyperplane containing $0$, hence an orthogonal involution. Since $\nu\in\Pi$, $\sigma$ acts as the identity on $\Pi^\perp=\mathrm{span}\{e_2,\dots,e_{N-1}\}$. Since $\tau\perp\nu$, $\sigma\tau=\tau$ and so $\sigma x\cdot\tau=x\cdot\tau$, while $\sigma x\cdot\nu=-x\cdot\nu$. The map $C$ acts as $-\mathrm{Id}$ on $\Pi$ and as the identity on $\Pi^\perp$, so it commutes with every linear map preserving $\Pi$ and $\Pi^\perp$, in particular with $\sigma$. Moreover $C$ changes the signs of both $x\cdot\nu$ and $x\cdot\tau$. The identities for the quadrants follow.

(ii) We have $e_N=\cos\varphi\,\nu+\sin\varphi\,\tau$. Indeed
\[
\cos\varphi(-\sin\varphi\,e_1+\cos\varphi\,e_N)+\sin\varphi(\cos\varphi\,e_1+\sin\varphi\,e_N)=e_N .
\]
Hence $x_N=x\cdot e_N=\cos\varphi\,(x\cdot\nu)+\sin\varphi\,(x\cdot\tau)$. Replacing $x$ by $\sigma x$ and using (i) gives the second formula.

(iii) If $x\cdot\nu\ge0$ and $x\cdot\tau\ge0$, then both terms in the formula for $x_N$ are non-negative, so $x_N\ge0$. Furthermore
\[
-x_N=-\cos\varphi\,(x\cdot\nu)-\sin\varphi\,(x\cdot\tau)\le(\sigma x)_N\le\cos\varphi\,(x\cdot\nu)+\sin\varphi\,(x\cdot\tau)=x_N ,
\]
i.e.\ $|(\sigma x)_N|\le x_N=|x_N|$. Since $|\sigma x|=|x|$, we get $|(\sigma x)'|^2=|x|^2-(\sigma x)_N^2\ge|x'|^2$ and $\sin\theta(\sigma x)=|(\sigma x)_N|/|x|\le\sin\theta(x)$.

(iv) If $x\cdot\tau=0$, then \eqref{eq:coords} gives $x_N=\cos\varphi\,(x\cdot\nu)=-(\sigma x)_N$. Hence $|(\sigma x)'|^2=|x|^2-x_N^2=|x'|^2$. If $x\cdot\nu=0$, then $\sigma x=x$ by definition.
\end{proof}

\begin{corollary}\label{cor:reflect}
Let $\Omega$ be admissible and $\varphi\in(0,\pi/2)$. Then
\[
\sigma_\varphi\big(\Omega\cap\overline{Q^+_\varphi}\big)\subset\Omega .
\]
Moreover, if $v:\R^N\to\R$ satisfies $v\circ g=v$ pointwise for every $g\in\G$, then $v(\sigma_\varphi x)=v(x)$ for every $x\in Z_\varphi$.
\end{corollary}

\begin{proof}
Let $x\in\Omega\cap\overline{W_\varphi}$. By Lemma~\ref{lem:geom}(iii), $|\sigma x|=|x|$ and $|(\sigma x)_N|\le|x_N|$, so $\sigma x\in\Omega$ by condition (b). Now let $x\in\Omega\cap C(\overline{W_\varphi})$. Then $y:=Cx\in\Omega\cap\overline{W_\varphi}$, because $\Omega$ is $C$-invariant. Hence $\sigma x=\sigma Cy=C\sigma y\in C\Omega=\Omega$. Since $\overline{Q^+_\varphi}=\overline{W_\varphi}\cup C(\overline{W_\varphi})$, the first claim follows.

For the second claim, if $x\in\tau^\perp$ then by Lemma~\ref{lem:geom}(iv) the points $x$ and $\sigma x$ have the same values of $|x'|$ and $|x_N|$. Since $O(N-1)$ acts transitively on the spheres of $\R^{N-1}$ and $O(1)$ changes the sign of $x_N$, there is $g\in\G$ with $\sigma x=gx$, so $v(\sigma x)=v(x)$. If $x\in\nu^\perp$, then $\sigma x=x$.
\end{proof}

The cone $\K$ is designed to encode, in a weak form, monotonicity from the poles towards the equator. Indeed, by Lemma~\ref{lem:geom}, for $x\in W_\varphi$ the reflected point $\sigma_\varphi x$ has the same distance from the origin and a smaller latitude. Thus condition~\eqref{eq:reflK}
 requires that the value of a function in $\K$ does not decrease when points are reflected towards the equatorial region. 

By Corollary~\ref{cor:reflect}, if $x\in W_\varphi\cap\Omega$ then $\sigma_\varphi x\in\Omega$, so \eqref{eq:reflK} compares values of $u$ inside $\Omega$. If $x\in W_\varphi\setminus\Omega$, then $u(x)=0$ and \eqref{eq:reflK} is trivial. 

We conclude this section by recording a few elementary properties of $\K$ that will be used in the sequel.

\begin{lemma}\label{lem:Kprop}
Let $\Omega$ be admissible.
\begin{enumerate}[label=\textup{(\alph*)}]
\item $\K$ is a convex cone, closed in $H^1_0(\Omega)$.
\item Every non-negative radial function of $H^1_0(\Omega)$ belongs to $\K$. 
\item If $u\in\K$ and $\varphi\in(0,\pi/2)$, then $u(\sigma_\varphi x)\ge u(x)$ for a.e.\ $x\in Q^+_\varphi$ and $u(\sigma_\varphi x)\le u(x)$ for a.e.\ $x\in Q^-_\varphi$.
\item If $u\in\K$ and $h:[0,\infty)\to[0,\infty)$ is non-decreasing and Lipschitz with $h(0)=0$, then $h(u)\in\K$.
\end{enumerate}
\end{lemma}

\begin{proof}
(a) All the conditions defining $\K$ are preserved by sums and by multiplication by non-negative constants. If $u_n\in\K$ and $u_n\to u$ in $H^1_0(\Omega)$, then, up to a subsequence, $u_n\to u$ a.e.\ in $\R^N$. Since $\sigma_\varphi$ and the elements of $\G$ preserve the Lebesgue measure, also $u_n\circ\sigma_\varphi\to u\circ\sigma_\varphi$ and $u_n\circ g\to u\circ g$ a.e. Passing to the limit in the pointwise inequalities gives $u\in\K$. A closed convex set is weakly closed.

(b) If $u$ is radial, then $u\circ g=u$ for every orthogonal map $g$, in particular for $g\in\G$ and $g=\sigma_\varphi$. Since $0\in\Omega$, there is a ball $B_\rho\subset\Omega$, and any non-negative, non-zero radial function in $C^\infty_c(B_\rho)$ belongs to $\K$.

(c) On $W_\varphi$ this is \eqref{eq:reflK}. If $x\in C(W_\varphi)$, then $y:=Cx\in W_\varphi$ and, since $C$ commutes with $\sigma_\varphi$ and $u\circ C=u$,
\[
u(\sigma_\varphi x)=u(\sigma_\varphi Cy)=u(C\sigma_\varphi y)=u(\sigma_\varphi y)\ge u(y)=u(Cy)=u(x).
\]
This holds for a.e.\ $x$, because $C$ preserves the measure. Since $Q^+_\varphi=W_\varphi\cup C(W_\varphi)$, the first inequality follows. If $x\in Q^-_\varphi$, then $y:=\sigma_\varphi x\in Q^+_\varphi$ and $u(x)=u(\sigma_\varphi y)\ge u(y)=u(\sigma_\varphi x)$.

(d) Since $h$ is Lipschitz with $h(0)=0$, $h(u)\in H^1_0(\Omega)$ (see e.g.\ \cite[Chapter II, Theorem A.1]{KS}). Clearly $h(u)\ge0$ and $h(u)$ is $\G$-invariant, and $u(\sigma x)\ge u(x)$ implies $h(u(\sigma x))\ge h(u(x))$ because $h$ is non-decreasing.
\end{proof}

\begin{lemma}\label{lem:monot}
Let $\Omega$ be admissible and $u\in\K\cap C(\Omega)$.
\begin{enumerate}[label=\textup{(\roman*)}]
\item For every $\varphi\in(0,\pi/2)$ and every $x\in\Omega\cap W_\varphi$ we have $u(\sigma_\varphi x)\ge u(x)$.
\item $u$ satisfies the monotonicity property \textup{(ii)} of Theorem~\ref{thm:main}.
\end{enumerate}
\end{lemma}

\begin{proof}
(i) By Corollary~\ref{cor:reflect}, $x\mapsto u(\sigma_\varphi x)-u(x)$ is a continuous function on the open set $\Omega\cap W_\varphi$. It is non-negative a.e.\ there, hence everywhere.

(ii) Let $e$, $r$, $\theta'<\theta$ be as in Theorem~\ref{thm:main}(ii), with $x_\theta\in\Omega$. Since $|x_{\theta'}|=r=|x_\theta|$ and $|(x_{\theta'})_N|=r\sin\theta'\le r\sin\theta=|(x_\theta)_N|$, condition (b) gives $x_{\theta'}\in\Omega$. By $O(N-1)$-invariance we may assume $e=e_1$. Put $\varphi:=(\theta+\theta')/2$. Since $0\le\theta'<\theta\le\pi/2$, we have $\varphi\in(0,\pi/2)$ and $\theta-\varphi=(\theta-\theta')/2\in(0,\pi/4]$. Hence
\[
x_\theta\cdot\nu_\varphi=r\sin(\theta-\varphi)>0,\qquad x_\theta\cdot\tau_\varphi=r\cos(\theta-\varphi)>0,
\]
so $x_\theta\in\Omega\cap W_\varphi$. By the description of $\sigma_\varphi$ in the plane $\Pi$, we have $\sigma_\varphi x_\theta=x_{2\varphi-\theta}=x_{\theta'}$. Part (i) gives $u(x_{\theta'})\ge u(x_\theta)$.
\end{proof}

%=====================================================================
\section{Weighted inequalities and compactness}\label{sec:emb}
%=====================================================================

The aim of this section is to show that the geometric conditions on the cone $\K$ imply the validity of a Sobolev inequality, in the spirit of Ni's paper (see \cite{Ni}). We begin by proving that invariance under the group $O(N-1)$ implies the existence of a weighted two-dimensional inequality.

\begin{proposition}\label{prop:cyl}
Let $N\ge3$ and $\kappa\ge0$. There is a constant $C_\kappa=C(N,\kappa)>0$ such that, for every bounded $O(N-1)$-invariant open set $\Omega\subset\R^N$ and every $O(N-1)$-invariant $\phi\in H^1_0(\Omega)$,
\begin{equation}\label{eq:cyl}
\int_\Omega|\phi|^{2^*+2\kappa}\,|x'|^{\kappa(N-2)}\dd x\le C_\kappa\,\|\nabla\phi\|_{L^2(\Omega)}^{2^*+2\kappa}.
\end{equation}
\end{proposition}

\begin{remark}\label{rem:scaling}
Inequality \eqref{eq:cyl} is invariant under the dilations $\phi_\lambda(x)=\lambda^{(N-2)/2}\phi(\lambda x)$. Indeed $\|\nabla\phi_\lambda\|_2=\|\nabla\phi\|_2$, and the change of variables $y=\lambda x$ gives
\[
\begin{aligned}
\int|\phi_\lambda|^{2^*+2\kappa}|x'|^{\kappa(N-2)}\dd x&=\lambda^{\frac{N-2}{2}(2^*+2\kappa)-N-\kappa(N-2)}\int|\phi|^{2^*+2\kappa}|y'|^{\kappa(N-2)}\dd y\\
&=\int|\phi|^{2^*+2\kappa}|y'|^{\kappa(N-2)}\dd y .
\end{aligned}
\]
The symmetry is essential when $\kappa>0$. If $\phi\in C^\infty_c(B_1)$ is radial and $x_0=e_1$, the functions $\lambda^{(N-2)/2}\phi(\lambda(x-x_0))$ have bounded Dirichlet integral, but the left-hand side of \eqref{eq:cyl} behaves like $\lambda^{\kappa(N-2)}\to\infty$. These functions are not $O(N-1)$-invariant. 
\end{remark}

The proof relies on the following elementary inequality in the half-plane.

\begin{lemma}\label{lem:GN}
Let $\Phi:[0,\infty)\times\R\to\R$ be Lipschitz with compact support and let $\beta\ge0$. Then
\begin{equation}\label{eq:GN}
\int_0^\infty\!\!\int_\R\Phi(s,t)^2\,s^{2\beta}\dd t\dd s\le\left(\int_0^\infty\!\!\int_\R|\nabla\Phi(s,t)|\,s^{\beta}\dd t\dd s\right)^2 .
\end{equation}
\end{lemma}

\begin{proof}
For every $t$ the function $s\mapsto\Phi(s,t)$ is Lipschitz and vanishes for large $s$. For every $s$ the function $t\mapsto\Phi(s,t)$ is Lipschitz and vanishes for large $|t|$. Hence, for a.e.\ $(s,t)$,
\[
|\Phi(s,t)|\le\int_s^\infty|\nabla\Phi(\tau,t)|\dd\tau,\qquad |\Phi(s,t)|\le\int_t^\infty|\nabla\Phi(s,\varsigma)|\dd\varsigma\le\int_\R|\nabla\Phi(s,\varsigma)|\dd\varsigma .
\]
Here the partial derivatives along a.e.\ line coincide a.e.\ with the corresponding components of $\nabla\Phi$, by Fubini's theorem. Multiplying the two inequalities and using $s^\beta\le\tau^\beta$ for $\tau\ge s$, we get
\[
\Phi(s,t)^2s^{2\beta}\le\Big(\int_0^\infty|\nabla\Phi(\tau,t)|\tau^\beta\dd\tau\Big)\Big(s^\beta\int_\R|\nabla\Phi(s,\varsigma)|\dd\varsigma\Big)=:a(t)\,b(s).
\]
Integrating over $(s,t)\in(0,\infty)\times\R$ and using Fubini's theorem,
\[
\int_0^\infty\!\!\int_\R\Phi^2s^{2\beta}\dd t\dd s\le\Big(\int_\R a(t)\dd t\Big)\Big(\int_0^\infty b(s)\dd s\Big)=\Big(\int_0^\infty\!\!\int_\R|\nabla\Phi|s^\beta\dd t\dd s\Big)^2 . \qedhere
\]
\end{proof}

\begin{proof}[Proof of Proposition~\ref{prop:cyl}]
\emph{Step 1.}
It is quite standard that $O(N-1)$-invariant $C^\infty_{c}(\Omega)$ functions are dense in $O(N-1)$-invariant functions of $H^1_0(\Omega)$, see \cite{HV}.  It is then enough to prove \eqref{eq:cyl} for Lipschitz $O(N-1)$-invariant functions with compact support in $\Omega$. Indeed if \eqref{eq:cyl} holds for each $\phi_n$, and $\phi_n\to \phi$ in $H^1_0(\Omega)$ and a.e. in $\Omega$, Fatou's lemma gives
\[
\int|\phi|^{2^*+2\kappa}|x'|^{\kappa(N-2)}\le\liminf_n\int|\phi_n|^{2^*+2\kappa}|x'|^{\kappa(N-2)}\le C_\kappa\lim_n\|\nabla\phi_n\|_2^{2^*+2\kappa}=C_\kappa\|\nabla\phi\|_2^{2^*+2\kappa}.
\]

Next,  if $\phi$ is a Lipschitz function with compact support in $\Omega$ set $\Phi(s,t):=\phi(se_1+te_N)$, $(s,t)\in[0,\infty)\times\R$. Then $\Phi$ is Lipschitz with compact support and $\phi(x)=\Phi(|x'|,x_N)$. For $\gamma\ge1$ the function $|\Phi|^\gamma$ is Lipschitz with compact support and $|\nabla|\Phi|^\gamma|=\gamma|\Phi|^{\gamma-1}|\nabla\Phi|$ a.e. Let $\beta\ge(N-2)/2$. Applying \eqref{eq:GN} to $|\Phi|^\gamma$ and then the Cauchy--Schwarz inequality,
\begin{align}
\int_0^\infty\!\!\int_\R|\Phi|^{2\gamma}s^{2\beta}\dd t\dd s
&\le\gamma^2\Big(\int_0^\infty\!\!\int_\R\big(|\nabla\Phi|\,s^{\frac{N-2}2}\big)\big(|\Phi|^{\gamma-1}s^{\beta-\frac{N-2}2}\big)\dd t\dd s\Big)^2\notag\\
&\le\gamma^2\Big(\int_0^\infty\!\!\int_\R|\nabla\Phi|^2s^{N-2}\dd t\dd s\Big)\Big(\int_0^\infty\!\!\int_\R|\Phi|^{2\gamma-2}s^{2\beta-(N-2)}\dd t\dd s\Big).\label{eq:step}
\end{align}
For integers $k\ge0$ define, using polar coordinates in $\R^{N-1}$,
\[
J_k:=\int_0^\infty\!\!\int_\R|\Phi|^{2^*+2k}s^{(k+1)(N-2)}\dd t\dd s=\frac1{\omega_{N-2}}\int_{\R^N}|\phi|^{2^*+2k}|x'|^{k(N-2)}\dd x.
\]
Also,
\[
D:=\int_0^\infty\!\!\int_\R|\nabla\Phi|^2s^{N-2}\dd t\dd s=\frac1{\omega_{N-2}}\|\nabla\phi\|_2^2 .
\]
For $k\ge1$ choose
\[
\gamma_k:=\frac N{N-2}+k\ge1,\qquad \beta_k:=\frac{(k+1)(N-2)}2\ge\frac{N-2}2 .
\]
Then $2\gamma_k=2^*+2k$, $2\beta_k=(k+1)(N-2)$, $2\gamma_k-2=2^*+2(k-1)$ and $2\beta_k-(N-2)=k(N-2)$. So \eqref{eq:step} reads
\[
J_k\le\gamma_k^2\,D\,J_{k-1}\qquad(k\ge1).
\]
All the quantities are finite because $\Phi$ is bounded with compact support. By induction,
\[
J_k\le\Big(\prod_{j=1}^k\gamma_j^2\Big)D^kJ_0 .
\]
By the Sobolev inequality $\mathcal S_N\|\phi\|_{2^*}^2\le\|\nabla\phi\|_2^2$ we have $J_0=\omega_{N-2}^{-1}\|\phi\|_{2^*}^{2^*}\le\omega_{N-2}^{-1}\mathcal S_N^{-N/(N-2)}\|\nabla\phi\|_2^{2^*}$. Therefore
\begin{equation}\label{eq:Ck}
\int_{\R^N}|\phi|^{2^*+2k}|x'|^{k(N-2)}\dd x=\omega_{N-2}J_k\le\Big(\prod_{j=1}^k\gamma_j^2\Big)\omega_{N-2}^{-k}\mathcal S_N^{-\frac N{N-2}}\ \|\nabla\phi\|_2^{2^*+2k}.
\end{equation}
This is \eqref{eq:cyl} for integer $\kappa=k$.

\emph{Step 2.} Let $\kappa>0$ be arbitrary, let $k\ge\kappa$ be an integer, and let $\lambda:=\kappa/k\in(0,1]$. Since
\[
2^*+2\kappa=(1-\lambda)2^*+\lambda(2^*+2k)\qquad\text{and}\qquad\kappa(N-2)=\lambda k(N-2),
\]
H\"older's inequality with exponents $1/(1-\lambda)$ and $1/\lambda$ gives
\begin{multline*}
\int|\phi|^{2^*+2\kappa}|x'|^{\kappa(N-2)}=\int\big(|\phi|^{2^*}\big)^{1-\lambda}\big(|\phi|^{2^*+2k}|x'|^{k(N-2)}\big)^{\lambda}\\
\le\|\phi\|_{2^*}^{2^*(1-\lambda)}\Big(\int|\phi|^{2^*+2k}|x'|^{k(N-2)}\Big)^\lambda .
\end{multline*}
By the Sobolev inequality and \eqref{eq:Ck}, the right-hand side is bounded by a constant times $\|\nabla\phi\|_2^{2^*(1-\lambda)+(2^*+2k)\lambda}=\|\nabla\phi\|_2^{2^*+2\kappa}$.
\end{proof}

Inequality \eqref{eq:reflK} is crucial in deriving the weighted Sobolev estimates for functions in $\K$.
Let
\[
S:=\{x:\ |x_N|<|x'|\},\qquad P^\pm:=\{x:\ \pm x_N>|x'|\}.
\]
$S$ is the open double cone of latitude less than $\pi/4$ around the equator, and $P^\pm$ are the two polar cones. The set $\R^N\setminus(S\cup P^+\cup P^-)=\{|x_N|=|x'|\}$ has measure zero, and on $S$ we have
\begin{equation}\label{eq:Scomp}
|x|^2=|x'|^2+x_N^2<2|x'|^2 .
\end{equation}

\begin{lemma}\label{lem:transfer}
Let $\Omega$ be admissible, $u\in\K$, $q>0$, and let $w:[0,\infty)\to[0,\infty)$ be measurable. Then
\[
\int_{\Omega}u^q\,w(|x|)\dd x\le 3\int_{\Omega\cap S}u^q\,w(|x|)\dd x .
\]
\end{lemma}

\begin{proof}
Take $\varphi=\pi/4$. Then $\nu_{\pi/4}=(e_N-e_1)/\sqrt2$ and $\tau_{\pi/4}=(e_1+e_N)/\sqrt2$, so $\sigma_{\pi/4}$ exchanges the coordinates $x_1$ and $x_N$, and
\[
W_{\pi/4}=\{x_N-x_1>0,\ x_N+x_1>0\}=\{x_N>|x_1|\}.
\]
Since $|x'|\ge|x_1|$, we have $P^+\subset W_{\pi/4}$. Moreover
\[
\sigma_{\pi/4}(P^+)=\{y:\ y_1>|(y_2,\dots,y_{N-1},y_N)|\}.
\]
For $y$ in this set, $|y'|\ge y_1>|y_N|$, so $\sigma_{\pi/4}(P^+)\subset S$. Using \eqref{eq:reflK} on $P^+\subset W_{\pi/4}$, the non-negativity of $u$, $|\sigma_{\pi/4}x|=|x|$ and the change of variables $y=\sigma_{\pi/4}x$, we obtain
\[
\int_{P^+}u(x)^q w(|x|)\dd x\le\int_{P^+}u(\sigma_{\pi/4}x)^qw(|x|)\dd x=\int_{\sigma_{\pi/4}(P^+)}u(y)^qw(|y|)\dd y\le\int_Su^qw(|y|)\dd y .
\]
Since $u\circ B=u$ and $B(P^+)=P^-$, the same bound holds on $P^-$. Adding the contributions of $S$, $P^+$ and $P^-$, and recalling that $u=0$ outside $\Omega$, gives the claim.
\end{proof}

\begin{remark}
The inequalities  \eqref{eq:reflK} cannot be dropped. Let $\phi\in C^\infty_c(B_1)$ be radial and non-negative, let $0<\rho<\sup_{x\in\Omega}|x_N|$ with $\pm\rho e_N\in\Omega$, and set
\[
v_\lambda(x):=\lambda^{\frac{N-2}2}\big(\phi(\lambda(x-\rho e_N))+\phi(\lambda(x+\rho e_N))\big).
\]
For $\lambda$ large, $v_\lambda\in H^1_{0,\G}(\Omega)$ and $\|\nabla v_\lambda\|_2$ is bounded. However $\int|x|^\alpha v_\lambda^q\dd x\approx2\rho^\alpha\lambda^{\frac{N-2}2q-N}\|\phi\|_q^q\to\infty$ for every $q>2^*$. So $H^1_{0,\G}(\Omega)$ does not embed into $L^q(\Omega,|x|^\alpha\dd x)$ for $q>2^*$: axial symmetry allows concentration at points of the axis. The functions $v_\lambda$ violate \eqref{eq:reflK} with $\varphi=\pi/4$, because their mass sits in the polar caps.
\end{remark}

The following proposition is central to our approach. By exploiting all the properties of functions $u\in\K$, we can extend Ni's inequality, which is valid only in the radially symmetric case \cite{Ni}.

\begin{proposition}\label{prop:embK}
Let $\Omega$ be admissible and $\kappa\ge0$. For every $u\in\K$,
\begin{equation}\label{eq:embK}
\int_\Omega u^{2^*+2\kappa}\,|x|^{\kappa(N-2)}\dd x\le 3\cdot2^{\frac{\kappa(N-2)}2}C_\kappa\,\|\nabla u\|_2^{2^*+2\kappa},
\end{equation}
where $C_\kappa$ is the constant of Proposition~\ref{prop:cyl}.
\end{proposition}

\begin{proof}
Apply Lemma~\ref{lem:transfer} with $q=2^*+2\kappa$ and $w(r)=r^{\kappa(N-2)}$. Then use \eqref{eq:Scomp} on $S$ and Proposition~\ref{prop:cyl}:
\[
\int_\Omega u^{2^*+2\kappa}|x|^{\kappa(N-2)}\le3\int_{S}u^{2^*+2\kappa}\,2^{\frac{\kappa(N-2)}2}|x'|^{\kappa(N-2)}\le3\cdot2^{\frac{\kappa(N-2)}2}C_\kappa\|\nabla u\|_2^{2^*+2\kappa}. \qedhere
\]
\end{proof}
We are now ready to prove Theorem \ref{cor:emb}.
\begin{proof}[Proof of Theorem \ref{cor:emb}]
If $2^*\le q\le q_\alpha$, let $\kappa:=(q-2^*)/2$ and $\alpha':=\kappa(N-2)=\frac{N-2}2q-N$. Then $0\le\alpha'\le\alpha$, because $q\le q_\alpha$ is equivalent to $\frac{N-2}2q-N\le\alpha$.  We let $R:=\sup_{x\in\Omega}|x|$.
Since $|x|^\alpha\le R^{\alpha-\alpha'}|x|^{\alpha'}$ in $\Omega$, the claim follows from \eqref{eq:embK}. If $1\le q<2^*$, H\"older's and Sobolev's inequalities give $\int_\Omega|x|^\alpha u^q\le R^\alpha|\Omega|^{1-q/2^*}\|u\|_{2^*}^q\le C\|\nabla u\|_2^q$.
\end{proof}

\begin{proposition}\label{prop:compact}
Let $\Omega$ be admissible, $\alpha>0$, $1\le q<q_\alpha$. If $(u_j)\subset\K$ and $u_j\rightharpoonup u$ weakly in $H^1_0(\Omega)$, then $u\in\K$ and
\[
\int_\Omega|x|^\alpha|u_j-u|^q\dd x\to0 .
\]
\end{proposition}

\begin{proof}
We have $u\in\K$ by Lemma~\ref{lem:Kprop}(a). Let $M:=\sup_j\|\nabla u_j\|_2+\|\nabla u\|_2<\infty$. By Rellich's theorem, $u_j\to u$ in $L^2(\Omega)$. It suffices to consider $q\ge2$; the case $1\le q<2$ follows from the case $q=2$ by H\"older's inequality. Fix $\delta\in(0,1)$.

\emph{Near the origin.} We claim that there is $\eta(\delta)\to0$ as $\delta\to0$ such that
\begin{equation}\label{eq:near0}
\int_{B_\delta}|x|^\alpha v^q\dd x\le\eta(\delta)\qquad\text{for every }v\in\K\text{ with }\|\nabla v\|_2\le M.
\end{equation}
If $q\le2^*$, H\"older's and Sobolev's inequalities give $\int_{B_\delta}|x|^\alpha v^q\le\delta^\alpha|B_\delta|^{1-q/2^*}\mathcal S_N^{-q/2}M^q$. If $2^*<q<q_\alpha$, let $\kappa:=(q-2^*)/2$ and $\alpha':=\kappa(N-2)\in(0,\alpha)$. By \eqref{eq:embK},
\[
\int_{B_\delta}|x|^\alpha v^q\le\delta^{\alpha-\alpha'}\int_\Omega|x|^{\alpha'}v^q\le\delta^{\alpha-\alpha'}\,3\cdot2^{\alpha'/2}C_\kappa M^q .
\]
In both cases \eqref{eq:near0} holds. Consequently
\[
\int_{B_\delta}|x|^\alpha|u_j-u|^q\le2^{q-1}\Big(\int_{B_\delta}|x|^\alpha u_j^q+\int_{B_\delta}|x|^\alpha u^q\Big)\le2^q\eta(\delta).
\]

\emph{Away from the origin.} Fix $q'>q$ with $q'\ge2^*$ and write $q'=2^*+2\kappa'$. For $v\in\K$ with $\|\nabla v\|_2\le M$, \eqref{eq:embK} gives
\[
\int_{\Omega\setminus B_\delta}v^{q'}\le\delta^{-\kappa'(N-2)}\int_\Omega v^{q'}|x|^{\kappa'(N-2)}\le\delta^{-\kappa'(N-2)}\,3\cdot2^{\frac{\kappa'(N-2)}2}C_{\kappa'}M^{q'}=:L(\delta)^{q'} .
\]
So $u_j-u$ is bounded in $L^{q'}(\Omega\setminus B_\delta)$ by $2L(\delta)$. Let $\lambda\in(0,1]$ be defined by $\frac1q=\frac\lambda2+\frac{1-\lambda}{q'}$. The interpolation inequality gives
\[
\|u_j-u\|_{L^q(\Omega\setminus B_\delta)}\le\|u_j-u\|_{L^2(\Omega)}^\lambda\big(2L(\delta)\big)^{1-\lambda}\to0 .
\]

\emph{Conclusion.} Since $|x|^\alpha\le R^\alpha$ on $\Omega$,
\[
\limsup_{j\to\infty}\int_\Omega|x|^\alpha|u_j-u|^q\le2^q\eta(\delta)+R^\alpha\limsup_{j\to\infty}\|u_j-u\|^q_{L^q(\Omega\setminus B_\delta)}=2^q\eta(\delta).
\]
Letting $\delta\to0$ concludes the proof.
\end{proof}

The final lemma of this section essentially follows from Hölder's inequality.
\begin{lemma}\label{lem:dual}
Let $\Omega$ be admissible, $\alpha>0$, $1\le p<p_\alpha$ and $r:=\frac{2N}{N+2}$. There is $C=C(N,\alpha,p,\Omega)$ such that
\[
\big\||x|^\alpha u^p\big\|_{L^r(\Omega)}\le C\|\nabla u\|_2^p\qquad\text{for every }u\in\K .
\]
In particular $|x|^\alpha u^p\in L^r(\Omega)\subset H^{-1}(\Omega)$, and $\big|\int_\Omega|x|^\alpha u^p\psi\big|\le C\mathcal S_N^{-1/2}\|\nabla u\|_2^p\|\nabla\psi\|_2$ for every $\psi\in H^1_0(\Omega)$.
\end{lemma}

\begin{proof}
Let $\kappa>0$ be a parameter, set $Q:=2^*+2\kappa$, and assume $\kappa$ is so large that $Q>pr$. Let $b:=\kappa(N-2)pr/Q$. H\"older's inequality with exponents $Q/(pr)$ and $Q/(Q-pr)$ gives
\[
\int_\Omega|x|^{\alpha r}u^{pr}=\int_\Omega\big(u^{pr}|x|^{b}\big)|x|^{\alpha r-b}\le\Big(\int_\Omega u^Q|x|^{\kappa(N-2)}\Big)^{\frac{pr}Q}\Big(\int_\Omega|x|^{e_\kappa}\Big)^{\frac{Q-pr}Q},
\]
where
\[
e_\kappa:=\frac{(\alpha r-b)Q}{Q-pr}=\frac{\alpha r\,2^*+\kappa\big(2\alpha r-(N-2)pr\big)}{2^*-pr+2\kappa}\ \xrightarrow[\kappa\to\infty]{}\ r\Big(\alpha-\frac{(N-2)p}2\Big).
\]
Now
\[
r\Big(\alpha-\frac{(N-2)p}2\Big)>-N\iff\frac{(N-2)p}2-\alpha<\frac Nr=\frac{N+2}2\iff p<\frac{N+2+2\alpha}{N-2}=p_\alpha .
\]
Hence, for $\kappa$ large enough, $e_\kappa>-N$ and $\int_\Omega|x|^{e_\kappa}\dd x\le\int_{B_R}|x|^{e_\kappa}\dd x<\infty$. The first factor is bounded by $\big(3\cdot2^{\kappa(N-2)/2}C_\kappa\|\nabla u\|_2^Q\big)^{pr/Q}$ by \eqref{eq:embK}. Thus $\int_\Omega|x|^{\alpha r}u^{pr}\le C\|\nabla u\|_2^{pr}$. The last assertion follows from H\"older's inequality ($1/r+1/2^*=1$) and the Sobolev inequality.
\end{proof}
%=====================================================================
\section{Invariance of \texorpdfstring{$\K$}{K} under the solution operator}\label{sec:inv}
%=====================================================================

For $\varphi\in(0,\pi/2)$ and a measurable function $v:\R^N\to\R$, we define the \emph{polarization}
\[
P_\varphi v(x):=
\begin{cases}
\min\{v(x),v(\sigma_\varphi x)\}&x\in Q^+_\varphi,\\[2pt]
\max\{v(x),v(\sigma_\varphi x)\}&x\in Q^-_\varphi,\\[2pt]
v(x)&x\in Z_\varphi .
\end{cases}
\]
Since $\sigma_\varphi(Q^+_\varphi)=Q^-_\varphi$, for $x\in Q^+_\varphi$ we have
\begin{equation}\label{eq:pairs}
P_\varphi v(x)=\min\{v(x),v(\sigma_\varphi x)\},\qquad P_\varphi v(\sigma_\varphi x)=\max\{v(x),v(\sigma_\varphi x)\}.
\end{equation}
So $P_\varphi$ rearranges the pair of values at $x$ and $\sigma_\varphi x$ so that the larger value sits in $Q^-_\varphi$, i.e.\ at the point closer to the equator when $x\in W_\varphi$. This is a two-point rearrangement in the sense of \cite{BrSo,VS}, adapted to the four quadrants $W_\varphi$, $\sigma_\varphi W_\varphi$, $CW_\varphi$, $\sigma_\varphi CW_\varphi$ in order to be compatible with the $\G$-symmetry. See Figure~\ref{fig:polarization-a} and \ref{fig:polarization-b}.

%\begin{figure}[ht]
%\centering
%\includegraphics[width=0.8\textwidth]{fig_polarization.pdf}
%\caption{Polarization in the plane $\Pi=\mathrm{span}\{e_1,e_N\}$ ($x'$ horizontal, $x_N$ vertical), for $\varphi=\pi/6$ and $\Omega$ an oblate ellipsoid.\protect\newline
%(a) The lines $\R\tau_\varphi$ and $\R\nu_\varphi$ and the four quadrants: $Q^+_\varphi=W_\varphi\cup CW_\varphi$ (blue) and $Q^-_\varphi=\sigma_\varphi W_\varphi\cup\sigma_\varphi CW_\varphi$ (orange).
%\protect\newline
%(b) The set $\Omega\cap Q^+_\varphi$ (blue) and its image under $\sigma_\varphi$ (orange contour {\Ma il contorno si confonde con il nero  di $\R\tau_\varphi$ and $\R\nu_\varphi$, meglio fare quei pezzo solo in orange}), which is contained in $\Omega$ by Corollary~\ref{cor:reflect}; a point $x\in W_\varphi$ is sent to $\sigma_\varphi x$, closer to the equator.
%\protect\newline
%(c) A non-negative $\mathcal G$-invariant function $v$ concentrated near the poles.
%\protect\newline
%(d) Its polarization $P_\varphi v$: on each pair $\{x,\sigma_\varphi x\}$ with $x\in Q^+_\varphi$ the larger value is moved to $\sigma_\varphi x\in Q^-_\varphi$, so mass is pushed towards the equator, while $\|\nabla P_\varphi v\|_2\le\|\nabla v\|_2$ (Lemma~\ref{lem:polariz}).}
%\label{fig:polarization}
%\end{figure}

\begin{figure}[ht]
\centering
\includegraphics[width=0.8\textwidth]{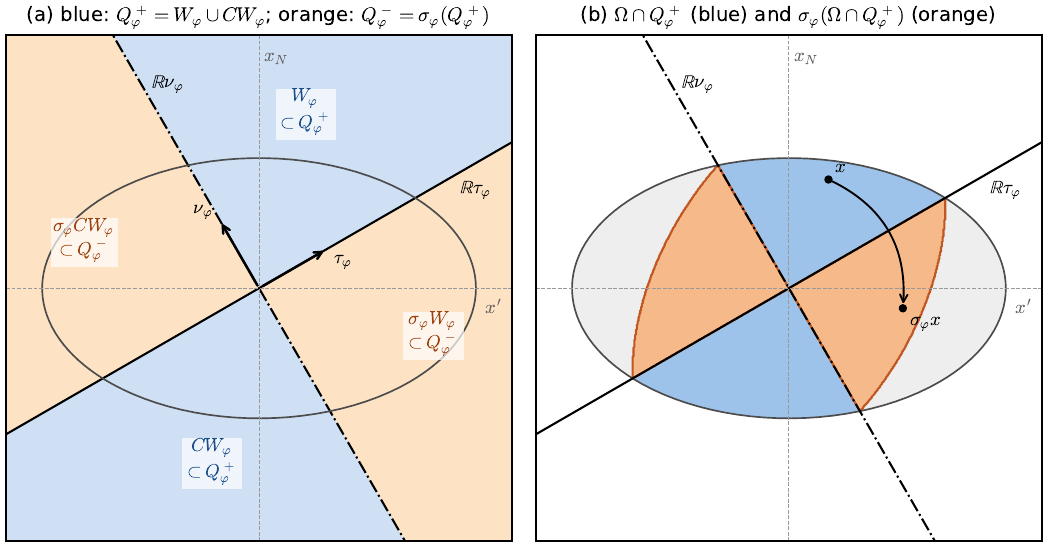}
\caption{The geometry of the polarization in the plane $\Pi=\mathrm{span}\{e_1,e_N\}$ ($x'$ horizontal, $x_N$ vertical), for $\varphi=\pi/6$ and $\Omega$ an oblate ellipsoid.
(a) The lines $\R\tau_\varphi$ and $\R\nu_\varphi$ and the four quadrants: $Q^+_\varphi=W_\varphi\cup CW_\varphi$ (blue) and $Q^-_\varphi=\sigma_\varphi W_\varphi\cup\sigma_\varphi CW_\varphi$ (orange).
(b) The set $\Omega\cap Q^+_\varphi$ (blue) and its image $\sigma_\varphi(\Omega\cap Q^+_\varphi)$ (orange), which is contained in $\Omega$ by Corollary~\ref{cor:reflect}; a point $x\in W_\varphi$ is sent to $\sigma_\varphi x$, closer to the equator.}
\label{fig:polarization-a}
\end{figure}

\begin{figure}[ht]
\centering
\includegraphics[width=0.8\textwidth]{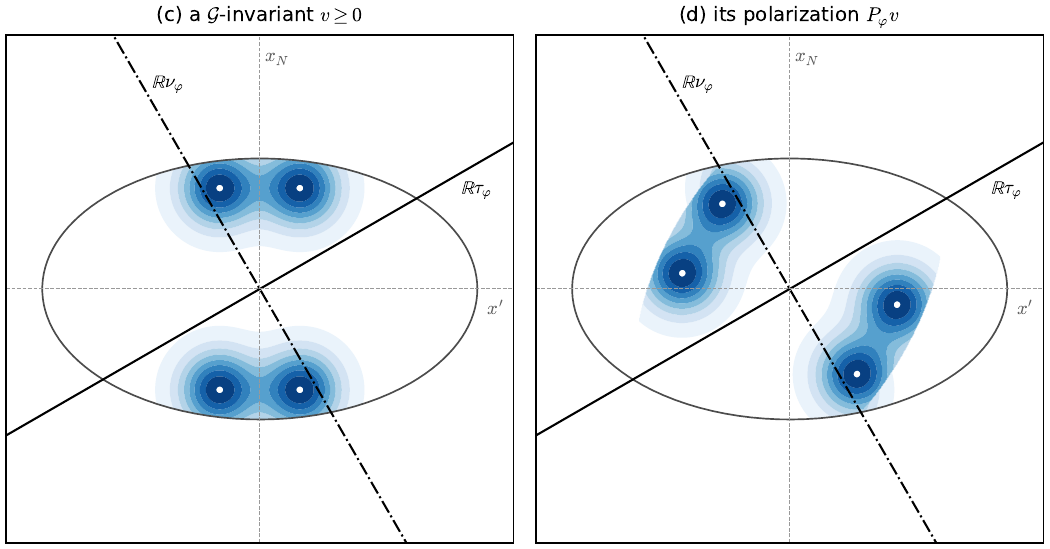}
\caption{(c) A non-negative $\mathcal G$-invariant function $v$ concentrated near the poles.
(d) Its polarization $P_\varphi v$: on each pair $\{x,\sigma_\varphi x\}$ with $x\in Q^+_\varphi$ the larger value is moved to $\sigma_\varphi x\in Q^-_\varphi$, so mass is pushed towards the equator, while $\|\nabla P_\varphi v\|_2\le\|\nabla v\|_2$ (Lemma~\ref{lem:polariz}). Same setting as in Figure~\ref{fig:polarization-a}.}
\label{fig:polarization-b}
\end{figure}

\begin{lemma}\label{lem:polariz}
Let $\Omega$ be admissible, $\varphi\in(0,\pi/2)$, $P=P_\varphi$, $\sigma=\sigma_\varphi$, and let $v\in H^1_{0,\G}(\Omega)$ with $v\ge0$.
\begin{enumerate}[label=\textup{(\alph*)}]
\item $Pv\in H^1_0(\Omega)$ and $\|\nabla Pv\|_2\le\|\nabla v\|_2$.
\item Let $f\in L^r(\R^N)$, $r=\frac{2N}{N+2}$, vanish outside $\Omega$ and satisfy $f(\sigma x)\ge f(x)$ for a.e.\ $x\in Q^+_\varphi$. Then $\int_\Omega fPv\ge\int_\Omega fv$.
\item $Pv=v$ a.e.\ if and only if $v(\sigma x)\ge v(x)$ for a.e.\ $x\in Q^+_\varphi$. 
\end{enumerate}
\end{lemma}

\begin{proof}
(a) \emph{Step 1: Lipschitz functions.} Let $\psi\ge0$ be a $\G$-invariant Lipschitz function with compact support $K_\psi\subset\Omega$.

\emph{$P\psi$ is Lipschitz.} On each of the four closed quadrants bounded by $\nu^\perp$ and $\tau^\perp$, $P\psi$ coincides with either $\min\{\psi,\psi\circ\sigma\}$ or $\max\{\psi,\psi\circ\sigma\}$. These are Lipschitz functions with constant $\mathrm{Lip}(\psi)$. On $Z_\varphi$ we have $\psi\circ\sigma=\psi$ by Corollary~\ref{cor:reflect}, so the minimum and the maximum both equal $\psi$ there, and $P\psi$ is continuous. A segment $[x,y]$ either lies in one of the hyperplanes or meets each of them at most once. Splitting it accordingly into at most three pieces, each contained in a closed quadrant, gives $|P\psi(x)-P\psi(y)|\le\mathrm{Lip}(\psi)|x-y|$.

\emph{$P\psi$ has compact support in $\Omega$.} If $x\in Q^+$ and $P\psi(x)\ne0$, then $\psi(x)>0$ because $\psi\ge0$, so $x\in K_\psi$. If $x\in Q^-$ and $P\psi(x)\ne0$, then $\psi(x)>0$ or $\psi(\sigma x)>0$; in the second case $\sigma x\in K_\psi\cap Q^+$. Hence
\[
\supp P\psi\subset K_\psi\cup\sigma\big(K_\psi\cap\overline{Q^+}\big),
\]
and the right-hand side is a compact subset of $\Omega$ by Corollary~\ref{cor:reflect}. In particular $P\psi\in H^1_0(\Omega)$.

\emph{$P\psi$ has the same Dirichlet integral as $\psi$.} Consider the disjoint sets
\[
\begin{gathered}
E_-:=\{x\in Q^+:\ \psi(x)<\psi(\sigma x)\},\qquad E_+:=\{x\in Q^+:\ \psi(x)>\psi(\sigma x)\},\\
E_=:=\{x\in Q^+:\ \psi(x)=\psi(\sigma x)\}.
\end{gathered}
\]
$E_-$ and $E_+$ are open. On $E_-\cup\sigma(E_-)$ we have $P\psi=\psi$. On $E_+\cup\sigma(E_+)$ we have $P\psi=\psi\circ\sigma$, by \eqref{eq:pairs}. On $E_=\cup\sigma(E_=)$ we have $P\psi=\psi=\psi\circ\sigma$, and $\nabla\psi=\nabla(\psi\circ\sigma)$ a.e.\ there. Since $|\nabla(\psi\circ\sigma)(x)|=|\nabla\psi(\sigma x)|$ and $\sigma$ preserves the measure of each of the three sets $E_+\cup\sigma(E_+)$,  $E_-\cup\sigma(E_-)$ and $E_=\cup\sigma(E_=)$ then,
\[
\int_{\R^N}|\nabla P\psi|^2=\int_{E_-\cup\sigma(E_-)}|\nabla\psi|^2+\int_{E_=\cup\sigma(E_=)}|\nabla\psi|^2+\int_{E_+\cup\sigma(E_+)}|\nabla\psi(\sigma x)|^2=\int_{\R^N}|\nabla\psi|^2 .
\]
\emph{Step 2: general $v$.} 
By density we have a sequence $v_n\in C^\infty_c(\Omega)$ of $\G$-invariant functions such that $v_n\to v$ in $H^1$. Since $v_n^+\to v^+=v$ we can assume that $v_n\geq 0$.
Since $\min$ and $\max$ are $1$-Lipschitz with respect to the sup norm on $\R^2$,
\[
|Pv_n(x)-Pv(x)|^2\le|v_n-v|^2(x)+|v_n-v|^2(\sigma x),\qquad\text{so}\qquad\|Pv_n-Pv\|_{L^2}^2\le2\|v_n-v\|_{L^2}^2\to0 .
\]
By Step 1, $(Pv_n)$ is bounded in $H^1_0(\Omega)$. Up to a subsequence it converges weakly in $H^1_0(\Omega)$, and the limit must be $Pv$. Hence $Pv\in H^1_0(\Omega)$ and
\[
\|\nabla Pv\|_2\le\liminf_n\|\nabla Pv_n\|_2=\lim_n\|\nabla v_n\|_2=\|\nabla v\|_2 .
\]

(b) Both $fv$ and $fPv$ are integrable, since $v,Pv\in L^{2^*}$ and $f\in L^r$ with $\frac1r+\frac1{2^*}=1$. By the change of variables $x\mapsto\sigma x$ on $Q^-=\sigma(Q^+)$ (and since $Z_\varphi$ is null), for every integrable $h$,
\[
\int_{\R^N}h=\int_{Q^+}\big(h(x)+h(\sigma x)\big)\dd x .
\]
Fix $x\in Q^+$, then by \eqref{eq:pairs},
\begin{multline*}
f(x)Pv(x)+f(\sigma x)Pv(\sigma x)-\big(f(x)v(x)+f(\sigma x)v(\sigma x)\big)=f(x)\min\{v(x),v(\sigma x)\}\\
+f(\sigma x)\max\{v(x),v(\sigma x)\}-f(x)v(x)-f(\sigma x)v(\sigma x) .
\end{multline*}
%\ta 
%\[
%f(x)Pv(x)+f(\sigma x)Pv(\sigma x)-\big(f(x)v(x)+f(\sigma x)v(\sigma x)\big)=f_1\min\{a,b\}+f_2\max\{a,b\}-f_1a-f_2b .
%\]
%This is $0$ if $a\le b$ and $(f_2-f_1)(a-b)\ge0$ if $a>b$. }
which vanishes if $v(x)\le v(\sigma x)$ and $(f(\sigma x)-f(x))(v(x)-v(\sigma x))\ge0$ if $v(x)>v(\sigma x)$. 
Integrating over $Q^+$ gives the claim.

(c) By \eqref{eq:pairs}, $Pv=v$ a.e.\ on $Q^+$ if and only if $v\le v\circ\sigma$ a.e.\ on $Q^+$, and in that case also $Pv(\sigma x)=\max\{v(x),v(\sigma x)\}=v(\sigma x)$ for a.e.\ $x\in Q^+$, i.e.\ $Pv=v$ a.e.\ on $Q^-$.
\end{proof}

\begin{proposition}[Invariance]\label{prop:inv}
Let $\Omega$ be admissible, $\alpha>0$, $1\le p<p_\alpha$ and $u\in\K$. Let $v\in H^1_0(\Omega)$ be the weak solution of
\begin{equation}\label{eq:lin}
-\Delta v=|x|^\alpha u^p\quad\text{in }\Omega,\qquad v=0\quad\text{on }\partial\Omega .
\end{equation}
Then $v\in\K$.
\end{proposition}

\begin{proof}
Let $f:=|x|^\alpha u^p$, extended by zero outside $\Omega$. By Lemma~\ref{lem:dual}, $f\in L^r(\Omega)\subset H^{-1}(\Omega)$. So $v$ exists, is unique, and is the unique minimizer on $H^1_0(\Omega)$ of the strictly convex functional
\[
I(z):=\frac12\int_\Omega|\nabla z|^2-\int_\Omega fz .
\]

\emph{$v\ge0$.} Since $f\ge0$, $I(|v|)\le I(v)$, hence $|v|=v$.

\emph{$v$ is $\G$-invariant.} For $g\in\G$, $v\circ g\in H^1_0(\Omega)$ because $g\Omega=\Omega$, and $f\circ g=f$. So $I(v\circ g)=I(v)$ and, by uniqueness of the minimizer, $v\circ g=v$.

\emph{Reflection inequalities.} Fix $\varphi\in(0,\pi/2)$. By Lemma~\ref{lem:Kprop}(c) and $|\sigma_\varphi x|=|x|$, we have $f(\sigma_\varphi x)=|x|^\alpha u(\sigma_\varphi x)^p\ge|x|^\alpha u(x)^p=f(x)$ for a.e.\ $x\in Q^+_\varphi$. By Lemma~\ref{lem:polariz}(a),(b), $P_\varphi v\in H^1_0(\Omega)$ and
\[
I(P_\varphi v)=\frac12\|\nabla P_\varphi v\|_2^2-\int fP_\varphi v\le\frac12\|\nabla v\|_2^2-\int fv=I(v).
\]
By uniqueness of the minimizer, $P_\varphi v=v$. By Lemma~\ref{lem:polariz}(c), $v(\sigma_\varphi x)\ge v(x)$ for a.e.\ $x\in Q^+_\varphi\supset W_\varphi$. Hence $v\in\K$.
\end{proof}

%=====================================================================
\section{Proof of Theorem~\ref{thm:main}}\label{sec:proof}
%=====================================================================

Throughout this section $\Omega$ is admissible, $\alpha>0$, $1<p<p_\alpha$, and
\[
F(u):=\int_\Omega|x|^\alpha|u|^{p+1}\dd x,\qquad \Sigma:=\{u\in\K:\ F(u)=1\},\qquad \mathsf S:=\inf_{u\in\Sigma}\int_\Omega|\nabla u|^2\dd x .
\]
Since $p+1<q_\alpha$, $F$ is finite on $\K$ by Theorem~\ref{cor:emb}.

\begin{lemma}\label{lem:min}
$0<\mathsf S<\infty$ and $\mathsf S$ is attained.
\end{lemma}

\begin{proof}
By Lemma~\ref{lem:Kprop}(b) there is $\psi\in\K$ with $F(\psi)>0$. Since $\K$ is a cone, $\psi/F(\psi)^{1/(p+1)}\in\Sigma$, so $\mathsf S<\infty$. By Theorem \ref{cor:emb} with $q=p+1$, for $u\in\Sigma$ we have $1=F(u)\le C\|\nabla u\|_2^{p+1}$, so $\mathsf S\ge C^{-2/(p+1)}>0$.

Let $(u_j)\subset\Sigma$ be a minimizing sequence. It is bounded in $H^1_0(\Omega)$, so up to a subsequence $u_j\rightharpoonup u$. By Proposition~\ref{prop:compact} with $q=p+1<q_\alpha$, $u\in\K$ and $u_j\to u$ in $L^{p+1}(\Omega,|x|^\alpha\dd x)$, so $F(u)=\lim F(u_j)=1$ and $u\in\Sigma$. By weak lower semicontinuity of the norm, $\|\nabla u\|_2^2\le\liminf\|\nabla u_j\|_2^2=\mathsf S$. Hence $u$ is a minimizer.
\end{proof}

\begin{lemma}[Variational inequality]\label{lem:VI}
Let $u\in\Sigma$ be a minimizer. Then for every $w\in\K$,
\begin{equation}\label{eq:VI}
\int_\Omega\nabla u\cdot\nabla(w-u)\dd x\ \ge\ \mathsf S\int_\Omega|x|^\alpha u^p(w-u)\dd x .
\end{equation}
\end{lemma}

\begin{proof}
Let $w\in\K$ and $u_t:=u+t(w-u)=(1-t)u+tw\in\K$ for $t\in[0,1]$. We have $0\le u_t\le u+w$. The map $t\mapsto F(u_t)$ is differentiable on $[0,1]$ with
\[
\frac{\dd}{\dd t}F(u_t)=(p+1)\int_\Omega|x|^\alpha u_t^p(w-u) .
\]
Indeed, the $t$-derivative of the integrand is bounded by $(p+1)|x|^\alpha(u+w)^p|w-u|$. This function is integrable by H\"older's inequality with exponents $\frac{p+1}p$ and $p+1$ with respect to the measure $|x|^\alpha\dd x$, since $u,w\in L^{p+1}(\Omega,|x|^\alpha\dd x)$ by Theorem~\ref{cor:emb}.  Since $F(u_0)=1$, we have $F(u_t)>0$ for $t$ small. As $\K$ is a cone, $u_t/F(u_t)^{1/(p+1)}\in\Sigma$, and so
\[
\mathcal R(t):=\frac{\|\nabla u_t\|_2^2}{F(u_t)^{2/(p+1)}}\ge\mathsf S=\mathcal R(0)\qquad\text{for $t\ge0$ small}.
\]
Therefore $\mathcal R'(0^+)\ge0$. Using $F(u)=1$ and $\|\nabla u\|_2^2=\mathsf S$,
\[
\mathcal R'(0^+)=2\int_\Omega\nabla u\cdot\nabla(w-u)-\mathsf S\cdot\frac2{p+1}\cdot(p+1)\int_\Omega|x|^\alpha u^p(w-u)\ \ge0,
\]
which is \eqref{eq:VI}.
\end{proof}

\begin{proposition}\label{prop:sol}
Let $u\in\Sigma$ be a minimizer. Then
\begin{equation}\label{eq:weak}
\int_\Omega\nabla u\cdot\nabla\psi\dd x=\mathsf S\int_\Omega|x|^\alpha u^p\psi\dd x\qquad\text{for every }\psi\in H^1_0(\Omega).
\end{equation}
\end{proposition}

\begin{proof}
Let $w\in H^1_0(\Omega)$ be the weak solution of $-\Delta w=\mathsf S|x|^\alpha u^p$ in $\Omega$. By Proposition~\ref{prop:inv} and since $\K$ is a cone, $w=\mathsf S\,v$ with $v\in\K$, so $w\in\K$. Testing the equation of $w$ with $w-u\in H^1_0(\Omega)$ (legitimate by Lemma~\ref{lem:dual}) gives
\[
\int_\Omega\nabla w\cdot\nabla(w-u)=\mathsf S\int_\Omega|x|^\alpha u^p(w-u).
\]
Choosing this $w$ in \eqref{eq:VI} and subtracting,
\[
\int_\Omega\nabla u\cdot\nabla(w-u)\ge\int_\Omega\nabla w\cdot\nabla(w-u),\qquad\text{i.e.}\qquad-\int_\Omega|\nabla(w-u)|^2\ge0 .
\]
Hence $w=u$, which is \eqref{eq:weak}.
\end{proof}

As recalled in the introduction, since the nonlinearity is in general supercritical, the boundedness of the solution is not automatic. However, the argument used in the following proposition is rather standard.

\begin{proposition}\label{prop:reg}
Let $U\in\K$ satisfy
\begin{equation}\label{eq:weakU}
\int_\Omega\nabla U\cdot\nabla\psi=\int_\Omega|x|^\alpha U^p\psi\qquad\text{for every }\psi\in H^1_0(\Omega).
\end{equation}
Then $U\in L^\infty(\Omega)\cap C^{2,\gamma}_{\rm loc}(\Omega)$. If moreover $\partial\Omega$ is Lipschitz, then $U\in C(\overline\Omega)$ and $U=0$ on $\partial\Omega$.
\end{proposition}

\begin{proof}
\emph{Step 1: $U\in L^q(\Omega)$ for every $q<\infty$.} Let $\beta\ge1$ and assume $U\in L^{2\beta}(\Omega)$. For $L>0$ define
\[
h_1(s):=s\min\{s,L\}^{\beta-1},\qquad h_2(s):=s\min\{s,L\}^{2(\beta-1)}\qquad(s\ge0).
\]
Both are non-decreasing, Lipschitz, and vanish at $0$. Hence $w:=h_1(U)\in\K$ by Lemma~\ref{lem:Kprop}(d), and $\psi:=h_2(U)\in H^1_0(\Omega)$. We have
\[
h_1'(s)^2=\begin{cases}\beta^2s^{2\beta-2}&s<L\\ L^{2\beta-2}&s>L\end{cases},\qquad
h_2'(s)=\begin{cases}(2\beta-1)s^{2\beta-2}&s<L\\ L^{2\beta-2}&s>L\end{cases}.
\]
Since $\beta^2/(2\beta-1)\le\beta^2$ and $1\le\beta^2$, we get $(h_1')^2\le\beta^2h_2'$ pointwise. Hence, by \eqref{eq:weakU} and $Uh_2(U)=h_1(U)^2=w^2$,
\[
\int_\Omega|\nabla w|^2=\int_\Omega h_1'(U)^2|\nabla U|^2\le\beta^2\int_\Omega h_2'(U)|\nabla U|^2=\beta^2\int_\Omega\nabla U\cdot\nabla\psi=\beta^2\int_\Omega|x|^\alpha U^{p-1}w^2 .
\]
For $M>0$ we split the last integral into $\{U>M\}$ and $\{U\le M\}$. On $\{U>M\}$ we use H\"older's inequality with exponents $\frac{p+1}{p-1}$, $\frac{p+1}2$ with respect to $|x|^\alpha\dd x$, together with Theorem~\ref{cor:emb} applied to $w\in\K$ with $q=p+1$:
\[
\int_{\{U>M\}}|x|^\alpha U^{p-1}w^2\le\epsilon_M^{\frac{p-1}{p+1}}\Big(\int_\Omega|x|^\alpha w^{p+1}\Big)^{\frac2{p+1}}\le\epsilon_M^{\frac{p-1}{p+1}}C_*\|\nabla w\|_2^2,\qquad \epsilon_M:=\int_{\{U>M\}}|x|^\alpha U^{p+1}.
\]
On $\{U\le M\}$,  letting $R=\sup_{x\in \Omega}|x|$, we use $w\le U^\beta$:
\[
\int_{\{U\le M\}}|x|^\alpha U^{p-1}w^2\le R^\alpha M^{p-1}\int_\Omega U^{2\beta}.
\]
Since $|x|^\alpha U^{p+1}\in L^1(\Omega)$, $\epsilon_M\to0$ as $M\to\infty$. Choose $M=M(\beta)$ such that $\beta^2C_*\epsilon_M^{(p-1)/(p+1)}\le\frac12$. As $\|\nabla w\|_2<\infty$ ($h_1$ is Lipschitz), we can absorb and obtain
\[
\|\nabla w\|_2^2\le2\beta^2R^\alpha M^{p-1}\int_\Omega U^{2\beta}.
\]
By the Sobolev inequality, $\big(\int_\Omega h_1(U)^{2^*}\big)^{2/2^*}\le\mathcal S_N^{-1}2\beta^2R^\alpha M^{p-1}\int_\Omega U^{2\beta}$. Letting $L\to\infty$, monotone convergence gives $U\in L^{2^*\beta}(\Omega)$. Starting from $\beta_0=1$ ($U\in L^2$) and iterating with $\beta_{k+1}=\frac{2^*}2\beta_k=\frac N{N-2}\beta_k\to\infty$, we get $U\in L^q(\Omega)$ for every $q<\infty$.

\emph{Step 2: regularity.} By Step 1, $f:=|x|^\alpha U^p\in L^q(\Omega)$ for every $q<\infty$. 
Taking $q>N/2$,  by classical arguments, see \cite{GT},  $U\in L^\infty(\Omega)$ and then $U\in C^{2,\gamma'}_{\rm loc}(\Omega)$.  Since $\Omega$ has Lipschitz continuous boundary by assumption, again by \cite{GT},  then $U\in C^{0,\gamma''}(\overline\Omega)$ for some $\gamma''>0$, with $U=0$ on $\partial\Omega$.
\end{proof}

\begin{proof}[Proof of Theorem~\ref{thm:main}]
Let $u\in\Sigma$ be a minimizer given by Lemma~\ref{lem:min}, and let
\[
U:=\mathsf S^{\frac1{p-1}}u\in\K .
\]
By Proposition~\ref{prop:sol}, for $\psi\in H^1_0(\Omega)$,
\[
\int\nabla U\cdot\nabla\psi=\mathsf S^{\frac1{p-1}}\mathsf S\int|x|^\alpha u^p\psi=\mathsf S^{\frac1{p-1}+1-\frac p{p-1}}\int|x|^\alpha U^p\psi=\int|x|^\alpha U^p\psi .
\]
So $U$ satisfies \eqref{eq:weakU}.  By Proposition~\ref{prop:reg}, $U\in C(\overline\Omega)\cap C^2(\Omega)$.
Moreover $U\ge0$, $U\not\equiv0$ and $-\Delta U\ge0$ in the connected open set $\Omega$, so the strong maximum principle gives $U>0$ in $\Omega$. Hence $U$ solves \eqref{eq:P}.

Property (i) holds because $U$ belongs to $\K\subset H^1_{0,\G}(\Omega)$ and $U$ is continuous. Property (ii) is Lemma~\ref{lem:monot}(ii). Property (iii) holds because $u=\mathsf S^{-1/(p-1)}U$ minimizes the Rayleigh quotient on $\K$: the quotient is $0$-homogeneous and $\K$ is a cone. 
\end{proof}

\begin{remark}[Energy characterization]\label{rem:energy}
The solution $U$ satisfies $\|\nabla U\|_2^2=F(U)=\mathsf S^{\frac{p+1}{p-1}}$. So its energy is
\[
J(U)=\frac12\int_\Omega|\nabla U|^2-\frac1{p+1}\int_\Omega|x|^\alpha U^{p+1}=\Big(\frac12-\frac1{p+1}\Big)\mathsf S^{\frac{p+1}{p-1}} .
\]
This is the least energy among the functions of $\K\setminus\{0\}$ lying on the Nehari manifold $\{\|\nabla v\|_2^2=F(v)\}$. In this sense $U$ is a ground state in $\K$.
\end{remark}

\begin{proof}[Proof of Corollary~\ref{cor:ball}]
Let 
\[\mathcal R(v):=\frac{\int_B|\nabla v|^2}{\left(\int_B|x|^\alpha|v|^{p+1}\right)^{2/(p+1)}}\] 
be the Rayleigh quotient of
Theorem~\ref{thm:main}(iii), and let $u_p$ be the positive radial solution of
\eqref{eq:P} in $B$, which is unique by \cite{NN}. If a minimizer $u$ of $\mathcal R$
on $\K\setminus\{0\}$ were radial, then $u_p$ itself would minimize
$\mathcal R$ on $\K\setminus\{0\}$. We construct
$\varphi$ with $u_p+\eps\varphi\in\K$ and $\mathcal R(u_p+\eps\varphi)<\mathcal R(u_p)$
for $\eps>0$ small, which gives a contradiction. Set
\[
Q(\psi):=\int_B|\nabla\psi|^2-p\int_B|x|^\alpha u_p^{p-1}\psi^2,\qquad \psi\in H^1_0(B).
\]
\smallskip\noindent\emph{Step 1: Construction of a test function $\psi$ such that $Q(\psi)<0$.}
\vskip0.1cm
Let
\[
Y(x):=N\frac{x_N^2}{|x|^2}-1 ,
\]
$Y$ is a spherical harmonic of degree $2$,
with $-\Delta_{S^{N-1}}Y=\lambda_2Y$, where $\lambda_j:=j(N+j-2)$, and
$\int_{S^{N-1}}Y=0$. Following \cite[Definition~3.1]{AG}, consider the singular
eigenvalue problem
\begin{equation}\label{eq:singeig}
-\Delta\phi-p|x|^\alpha u_p^{p-1}\phi=\hat\Lambda\,\frac{\phi}{|x|^2}\quad\text{in }B\setminus\{0\},
\qquad \phi\in H^1_{0,\rm rad}(B),
\end{equation}
and let $\hat\Lambda_1^{\rm rad}$ be its first eigenvalue and
$\hat\phi_1$ a corresponding eigenfunction with $\hat\phi_1\ge0$.
By \cite[(3.3) and Theorem~1.1]{AG}, the Morse index of $u_p$ is the number of
negative values $\hat\Lambda_1^{\rm rad}+\lambda_j$, each counted with the multiplicity
$N_j$ of $\lambda_j$ since $u_p$ has radial Morse index one, see \cite[Proposition~3.2]{AG}. Moreover, for $\alpha>2$ there is $\bar p\in(1,p_\alpha)$ such that
for $p\in[\bar p,p_\alpha)$ this Morse index equals $\sum_{j=0}^kN_j$ with $k\ge2$. In particular
\begin{equation}\label{eq:negdir}
\hat\Lambda_1^{\rm rad}+\lambda_2<0 .
\end{equation}
For a radial $\phi\in H^1_0(B)$, writing $r=|x|$ and using $|\nabla Y|^2=|x|^{-2}|\nabla_{S^{N-1}}Y|^2$
and $\int_{S^{N-1}}|\nabla_{S^{N-1}}Y|^2=\lambda_2\int_{S^{N-1}}Y^2$, we get
\begin{equation}\label{eq:QphiY}
Q(\phi Y)=\Big(\int_{S^{N-1}}Y^2\Big)\int_0^1\Big(\phi'^2+\frac{\lambda_2}{r^2}\phi^2
-p\,r^\alpha u_p^{p-1}\phi^2\Big)r^{N-1}\,dr .
\end{equation}
Testing \eqref{eq:singeig} with $\hat\phi_1$ and inserting the result in
\eqref{eq:QphiY} gives, by \eqref{eq:negdir},
\[
Q(\hat\phi_1Y)=\big(\hat\Lambda_1^{\rm rad}+\lambda_2\big)\int_B\frac{\hat\phi_1^2Y^2}{|x|^2}\,dx<0 .
\]
Since $N\ge3$, non-negative radial functions of $C^\infty_c(B\setminus\{0\})$ are dense
in the non-negative functions of $H^1_{0,\rm rad}(B)$. By the continuity of
$\phi\mapsto Q(\phi Y)$ we find a radial $w\in C^\infty_c(B\setminus\{0\})$ with $w\ge0$
and $Q(wY)<0$.

\smallskip\noindent\emph{Step 2: $u_p+\eps\varphi\in\K$ for $\eps>0$ small with $\varphi:=-wY$.}
\vskip0.1cm
We have that $u_p+\eps\varphi\in H^1_0(B) $ and is
$\mathcal G$-invariant, because $w$ is radial and $Y$  is
$\mathcal G$-invariant. Moreover by construction $u_p\ge c>0$ on $\operatorname{supp}w$. 

Next since $Y$ is non-decreasing in $|x_N|$ for fixed $|x|$, Lemma~\ref{lem:geom} shows that $u_p+\varepsilon\varphi\in\K$ for $\varepsilon>0$ small.

\smallskip\noindent\emph{Step 3: second order expansion.}
\vskip0.1cm
Since $u_p$ is radial and $\int_{S^{N-1}}Y=0$,
\begin{equation}\label{eq:orth}
\int_B\nabla u_p\cdot\nabla\varphi=\int_B|x|^\alpha u_p^p\varphi
=-\int_0^1 r^{\alpha}u_p(r)^p\,w(r)\,r^{N-1}\,dr\int_{S^{N-1}}Y=0 .
\end{equation}
Let $m:=\|\nabla u_p\|_2^2=F(u_p)$. By \eqref{eq:orth},
\[
\|\nabla(u_p+\eps\varphi)\|_2^2=m+\eps^2\|\nabla\varphi\|_2^2 .
\]
Since $u_p\ge c>0$ on the support of $\varphi$, a Taylor expansion together with \eqref{eq:orth} gives
\[
F(u_p+\eps\varphi)=m+\eps^2\,\frac{p(p+1)}2\int_B|x|^\alpha u_p^{p-1}\varphi^2+o(\eps^2).
\]
Hence
\[
\mathcal R(u_p+\eps\varphi)
=m^{\frac{p-1}{p+1}}\Big(1+\frac{\eps^2}m\|\nabla\varphi\|_2^2\Big)
\Big(1-\frac{\eps^2p}m\int_B|x|^\alpha u_p^{p-1}\varphi^2\Big)+o(\eps^2)
=\mathcal R(u_p)\Big(1+\eps^2\,\frac{Q(\varphi)}{F(u_p)}\Big)+o(\eps^2).
\]
Since $Q(\varphi)=Q(wY)<0$, we get $\mathcal R(u_p+\eps\varphi)<\mathcal R(u_p)$ for
$\eps>0$ small, with $u_p+\eps\varphi\in\K\setminus\{0\}$ by Step~2. This contradicts
the minimality of $u_p$.
\end{proof}

%=====================================================================
\section{Examples of admissible domains \texorpdfstring{$\Omega$}{Omega} }\label{sec:examples}
%=====================================================================

In this section we give a necessary condition for admissibility of a domain $\Omega$ in the sense of Definition \ref{def:adm} and a list of examples and non-examples with detailed computations. Recall that, for $x\ne0$, the latitude is $\theta(x)=\arcsin(|x_N|/|x|)\in[0,\pi/2]$.   Several examples are drawn in Figure~\ref{fig:admissible}. Three-dimensional images of some of these domains are shown in Figure~\ref{fig:3d}.

For admissible domains which are not star-shaped with respect to $0$ (Example~\ref{ex:nonstar}) the Pohozaev identity gives no information, and existence for $p\ge p_\alpha$ is open. For such $p$ the embedding of Proposition~\ref{prop:compact} fails, since concentration at the origin is allowed.

\vskip0.5cm
\begin{center}
\begin{minipage}{\textwidth}
\centering

\begin{tikzpicture}[x=1.55cm,y=1.55cm,
  panel/.style={draw=black,line width=.45pt},
  axes/.style={gray!60,densely dotted,line width=.45pt},
  guide/.style={gray!75,densely dotted,line width=.55pt},
  admfill/.style={fill=blue!15},
  admedge/.style={draw=blue!65!black,line width=.85pt,line join=round},
  slice/.style={draw=red,line width=1.6pt,line cap=butt},
  badfill/.style={fill=red!13},
  badedge/.style={draw=red!70!black,line width=.85pt},
  badslice/.style={draw=red!70!black,line width=1.6pt,line cap=butt},
  title/.style={font=\small}
]
% Local panel coordinates: x in [-1.45,1.45], y in [-1.02,1.02].
% Column spacing = 3.25, row spacing = 2.90.
% In every panel: (1) fill, (2) axes and guide circle |x|=r,
% (3) slices Omega cap {|x|=r}, clipped to Omega, (4) boundary on top.

% ------------------------------------------------------------
% 1. Oblate ellipsoid E_{a,b}: x'^2/a^2 + x_N^2/b^2 < 1   (here E_{2,1} scaled)
%    slice endpoint: sin^2(th) = (1/r^2-1/a^2)/(1/b^2-1/a^2)
% ------------------------------------------------------------
\begin{scope}[shift={(0,2.60)}]
  \pgfmathsetmacro{\ea}{1.28}\pgfmathsetmacro{\eb}{0.64}\pgfmathsetmacro{\er}{0.95}
  \pgfmathsetmacro{\eth}{asin(sqrt((1/(\er*\er)-1/(\ea*\ea))/(1/(\eb*\eb)-1/(\ea*\ea))))}
  \def\dom{(0,0) ellipse[x radius=\ea,y radius=\eb]}
  \node[title] at (0,1.18) {oblate ellipsoid of Example~\ref{ex:ellipsoid}};
  \draw[panel] (-1.45,-1.02) rectangle (1.45,1.02);
  \path[admfill] \dom;
  \draw[axes] (-1.45,0)--(1.45,0) (0,-1.02)--(0,1.02);
  \draw[guide] (0,0) circle[radius=\er];
  \begin{scope}
    \clip \dom;
    \draw[slice] (-\eth:\er) arc[start angle=-\eth,end angle=\eth,radius=\er];
    \draw[slice] (180-\eth:\er) arc[start angle=180-\eth,end angle=180+\eth,radius=\er];
  \end{scope}
  \draw[admedge] \dom;
  \fill (0,0) circle[radius=.02];
\end{scope}

% ------------------------------------------------------------
% 2. Lens L_{R,c} = B_R(c e_N) cap B_R(-c e_N)   (here L_{1.3,0.7} scaled by 1.1)
%    half width w = sqrt(R^2-c^2), half height R-c;
%    slice endpoint: sin(th) = (R^2-c^2-r^2)/(2cr)   (need R-c < r < w)
% ------------------------------------------------------------
\begin{scope}[shift={(3.25,2.60)}]
  \pgfmathsetmacro{\lR}{1.43}\pgfmathsetmacro{\lc}{0.77}\pgfmathsetmacro{\lr}{0.92}
  \pgfmathsetmacro{\lw}{sqrt(\lR*\lR-\lc*\lc)}
  \pgfmathsetmacro{\lal}{atan(\lc/\lw)}
  \pgfmathsetmacro{\lth}{asin((\lR*\lR-\lc*\lc-\lr*\lr)/(2*\lc*\lr))}
  % upper arc: circle centred at (0,-c); lower arc: circle centred at (0,c)
  \def\dom{(\lw,0) arc[start angle=\lal,end angle=180-\lal,radius=\lR]
                   arc[start angle=180+\lal,end angle=360-\lal,radius=\lR] -- cycle}
  \node[title] at (0,1.18) {lens of Example~\ref{ex:lens}};
  \draw[panel] (-1.45,-1.02) rectangle (1.45,1.02);
  \path[admfill] \dom;
  \draw[axes] (-1.45,0)--(1.45,0) (0,-1.02)--(0,1.02);
  \draw[guide] (0,0) circle[radius=\lr];
  \begin{scope}
    \clip \dom;
    \draw[slice] (-\lth:\lr) arc[start angle=-\lth,end angle=\lth,radius=\lr];
    \draw[slice] (180-\lth:\lr) arc[start angle=180-\lth,end angle=180+\lth,radius=\lr];
  \end{scope}
  \draw[admedge] \dom;
  \fill (0,0) circle[radius=.02];
\end{scope}

% ------------------------------------------------------------
% 3. Double paraboloid |x_N| < b(1-|x'|^2/a^2), admissible iff b <= a/sqrt(2)
%    slice endpoint: u=s^2 solves (b^2/a^4)u^2 + (1-2b^2/a^2)u + b^2-r^2 = 0,
%    th = atan2(h(s),s)   (need b < r < a)
% ------------------------------------------------------------
\begin{scope}[shift={(6.50,2.60)}]
  \pgfmathsetmacro{\pa}{1.32}\pgfmathsetmacro{\pb}{0.55}\pgfmathsetmacro{\pr}{0.85}
  \pgfmathsetmacro{\pk}{\pb*\pb/(\pa*\pa*\pa*\pa)}
  \pgfmathsetmacro{\pB}{1-2*\pb*\pb/(\pa*\pa)}
  \pgfmathsetmacro{\pu}{(-\pB+sqrt(\pB*\pB-4*\pk*(\pb*\pb-\pr*\pr)))/(2*\pk)}
  \pgfmathsetmacro{\pth}{atan2(\pb*(1-\pu/(\pa*\pa)),sqrt(\pu))}
  \def\dom{plot[domain=-\pa:\pa,samples=121] (\x,{\pb*(1-(\x*\x)/(\pa*\pa))})
           -- plot[domain=\pa:-\pa,samples=121] (\x,{-\pb*(1-(\x*\x)/(\pa*\pa))}) -- cycle}
  \node[title] at (0,1.18) {double paraboloid of Example~\ref{ex:saucer}};
  \draw[panel] (-1.45,-1.02) rectangle (1.45,1.02);
  \path[admfill] \dom;
  \draw[axes] (-1.45,0)--(1.45,0) (0,-1.02)--(0,1.02);
  \draw[guide] (0,0) circle[radius=\pr];
  \begin{scope}
    \clip \dom;
    \draw[slice] (-\pth:\pr) arc[start angle=-\pth,end angle=\pth,radius=\pr];
    \draw[slice] (180-\pth:\pr) arc[start angle=180-\pth,end angle=180+\pth,radius=\pr];
  \end{scope}
  \draw[admedge] \dom;
  \fill (0,0) circle[radius=.02];
\end{scope}

% ------------------------------------------------------------
% 4. Collar: B_b cup {x in B_a : theta(x) < theta_0}
%    slice for b < r < a: exactly {theta < theta_0}
% ------------------------------------------------------------
\begin{scope}[shift={(0,-0.30)}]
  \pgfmathsetmacro{\ca}{1.30}\pgfmathsetmacro{\cb}{0.70}\pgfmathsetmacro{\cr}{1.00}
  \pgfmathsetmacro{\ct}{35}
  \def\dom{(\ct:\cb) -- (\ct:\ca) arc[start angle=\ct,end angle=-\ct,radius=\ca]
           -- (-\ct:\cb) arc[start angle=-\ct,end angle=-180+\ct,radius=\cb]
           -- (180+\ct:\ca) arc[start angle=180+\ct,end angle=180-\ct,radius=\ca]
           -- (180-\ct:\cb) arc[start angle=180-\ct,end angle=\ct,radius=\cb] -- cycle}
  \node[title] at (0,1.18) {collar of Example~\ref{ex:collar}};
  \draw[panel] (-1.45,-1.02) rectangle (1.45,1.02);
  \path[admfill] \dom;
  \draw[axes] (-1.45,0)--(1.45,0) (0,-1.02)--(0,1.02);
  \draw[guide] (0,0) circle[radius=\cr];
  \begin{scope}
    \clip \dom;
    \draw[slice] (-\ct:\cr) arc[start angle=-\ct,end angle=\ct,radius=\cr];
    \draw[slice] (180-\ct:\cr) arc[start angle=180-\ct,end angle=180+\ct,radius=\cr];
  \end{scope}
  \draw[admedge] \dom;
  \fill (0,0) circle[radius=.02];
\end{scope}

% ------------------------------------------------------------
% 5. Non star-shaped: B_c minus T, T = {a <= |x| <= b, theta(x) >= theta_0}
%    slice for a < r < b: exactly {theta < theta_0}
% ------------------------------------------------------------
\begin{scope}[shift={(3.25,-0.30)}]
  \pgfmathsetmacro{\na}{0.50}\pgfmathsetmacro{\nb}{0.78}\pgfmathsetmacro{\nc}{0.97}
  \pgfmathsetmacro{\nr}{0.64}\pgfmathsetmacro{\nt}{50}
  \def\hole#1{(#1+\nt:\na) -- (#1+\nt:\nb) arc[start angle=#1+\nt,end angle=#1+180-\nt,radius=\nb]
              -- (#1+180-\nt:\na) arc[start angle=#1+180-\nt,end angle=#1+\nt,radius=\na] -- cycle}
  \def\dom{(0,0) circle[radius=\nc] \hole{0} \hole{180}}
  \node[title] at (0,1.18) {non star-shaped of Example~\ref{ex:nonstar}};
  \draw[panel] (-1.45,-1.02) rectangle (1.45,1.02);
  \path[admfill,even odd rule] \dom;
  \draw[axes] (-1.45,0)--(1.45,0) (0,-1.02)--(0,1.02);
  \draw[guide] (0,0) circle[radius=\nr];
  \begin{scope}[even odd rule]
    \clip \dom;
    \draw[slice] (-\nt:\nr) arc[start angle=-\nt,end angle=\nt,radius=\nr];
    \draw[slice] (180-\nt:\nr) arc[start angle=180-\nt,end angle=180+\nt,radius=\nr];
  \end{scope}
  \draw[admedge] \dom;
  \fill (0,0) circle[radius=.02];
\end{scope}

% ------------------------------------------------------------
% 6. Cylinder {|x'|<a, |x_N|<l} (not admissible), with a < r < sqrt(a^2+l^2)
%    slice: acos(a/r) < theta < asin(l/r), away from the equator
% ------------------------------------------------------------
% ------------------------------------------------------------
% 6. Cylinder (not admissible)
% ------------------------------------------------------------
\begin{scope}[shift={(6.50,-0.30)}]
  \pgfmathsetmacro{\ya}{0.84}\pgfmathsetmacro{\yl}{0.55}\pgfmathsetmacro{\yr}{0.95}
  \pgfmathsetmacro{\ylo}{acos(\ya/\yr)}\pgfmathsetmacro{\yhi}{asin(\yl/\yr)}
  \def\dom{(-\ya,-\yl) rectangle (\ya,\yl)}
  \node[title] at (0,1.18) {cylinder of Example~\ref{ex:cyl}};
  \draw[panel] (-1.45,-1.02) rectangle (1.45,1.02);
  \path[admfill] \dom;
  \draw[axes] (-1.45,0)--(1.45,0) (0,-1.02)--(0,1.02);
  \draw[guide] (0,0) circle[radius=\yr];
  \begin{scope}
    \clip \dom;
    \foreach \s in {0,180}{
      \draw[slice] (\s+\ylo:\yr) arc[start angle=\s+\ylo,end angle=\s+\yhi,radius=\yr];
      \draw[slice] (\s-\yhi:\yr) arc[start angle=\s-\yhi,end angle=\s-\ylo,radius=\yr];
    }
  \end{scope}
  \draw[admedge] \dom;
  \fill (0,0) circle[radius=.02];
\end{scope}
\end{tikzpicture}
\captionof{figure}{Sections of admissible domains in the plane $\Pi=\mathrm{span}\{e_1,e_N\}$ ($x'$ horizontal, $x_N$ vertical): the oblate ellipsoid of Example~\ref{ex:ellipsoid}, the lens of Example~\ref{ex:lens}, a double paraboloid with $b<a/\sqrt2$ (Example~\ref{ex:saucer}), the collar domain (Example~\ref{ex:collar}) and the non star-shaped domain (Example~\ref{ex:nonstar}). On the dotted circle $|x|=r$, the thick arc is $\Omega\cap\partial B_r$: for admissible domains it is a zone around the equator, as required by condition~(b). For the cylinder (Example~\ref{ex:cyl}) the slice $\Omega\cap\partial B_r$ misses the equator, so condition~(b) fails.}
\label{fig:admissible}
\end{minipage}
\end{center}

\begin{figure}[ht]
\centering
\includegraphics[width=0.8\textwidth]{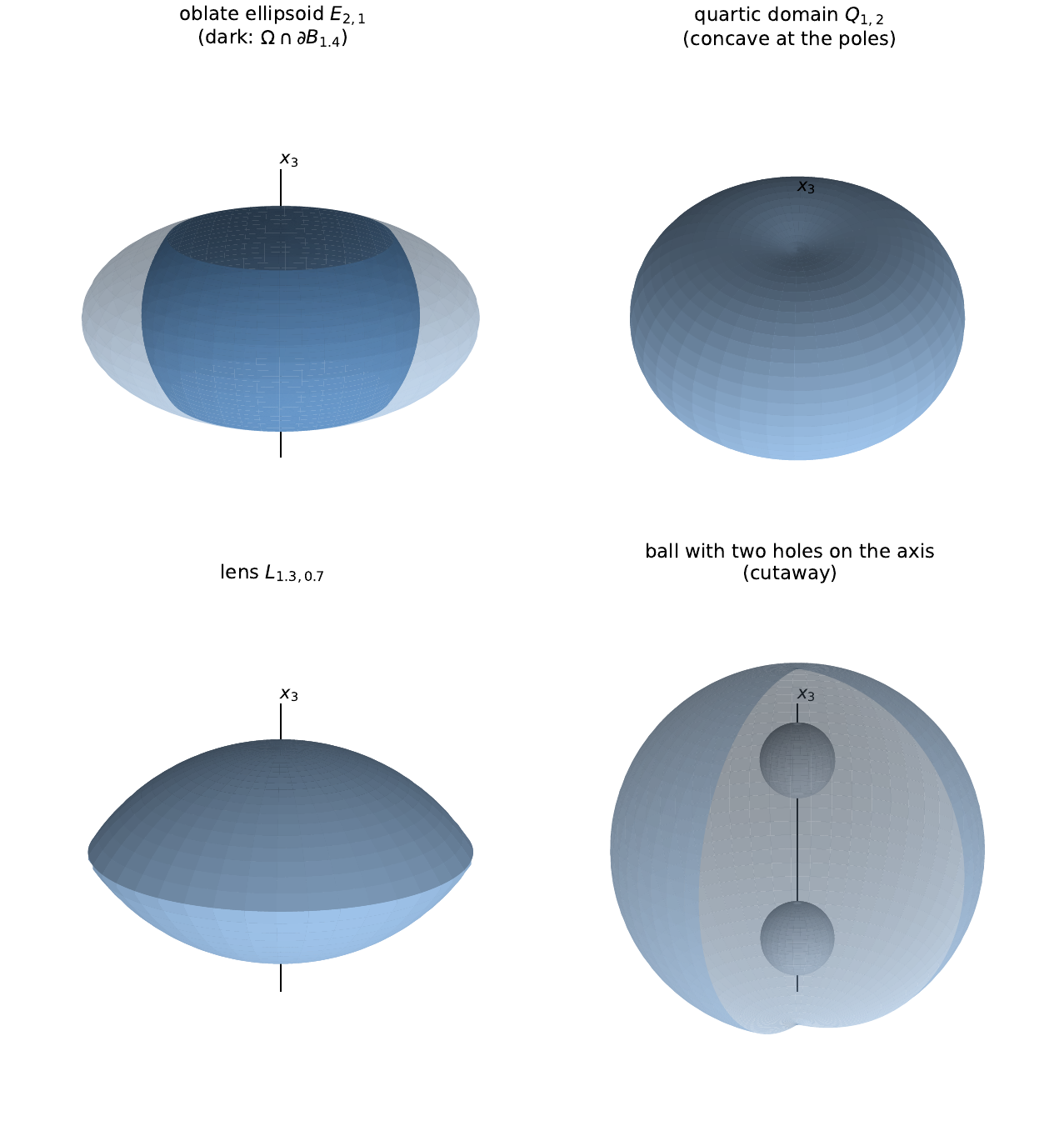}
\caption{Some admissible domains in $\R^3$, which are invariant under rotations around the $x_3$-axis and under $x_3\mapsto-x_3$. Top left: the oblate ellipsoid $E_{2,1}$ of Example~\ref{ex:ellipsoid}; the dark surface is $\Omega\cap\partial B_{1.4}$, a zone around the equator, as required by condition (b). Top right: the quartic domain $Q_{1,2}$ of Example~\ref{ex:quartic}, which is star-shaped but concave at the poles. Bottom left: the lens $L_{1.3,0.7}$ of Example~\ref{ex:lens}, with an edge along the equator. Bottom right: a cutaway view of the ball with two holes on the axis of Example~\ref{ex:holes}.}
\label{fig:3d}
\end{figure}

We start with a necessary condition.

\begin{lemma}\label{lem:disc}
Let $\Omega$ be admissible and $R_\Omega:=\sup_{x\in\Omega}|x|$. Then $\Omega$ contains the equatorial disc
\[
D_\Omega:=\{x:\ x_N=0,\ |x'|<R_\Omega\}.
\]
In particular, the points of $\overline\Omega$ at maximal distance from the origin include the whole equatorial $(N-2)$-sphere $\{x_N=0,\ |x'|=R_\Omega\}$.
\end{lemma}

\begin{proof}
The set $I:=\{|x|:\ x\in\Omega\}$ is connected, since $\Omega$ is connected, and contains $0$. It is open in $[0,\infty)$: if $x\in\Omega\setminus\{0\}$, the points $tx$ with $t$ close to $1$ lie in $\Omega$. Since $\sup I=R_\Omega$, we get $I=[0,R_\Omega)$. Now let $0\le s<R_\Omega$, pick $x\in\Omega$ with $|x|=s$, and let $e'\in\R^{N-1}$ be a unit vector. The point $y=(se',0)$ satisfies $|y|=|x|$ and $|y_N|=0\le|x_N|$, so $y\in\Omega$ by condition (b). Hence $D_\Omega\subset\Omega$.
\end{proof}

\subsection{Radial graphs}

\begin{lemma}\label{lem:polar}
Let $g:[0,\pi/2]\to(0,\infty)$ be lower semicontinuous and non-increasing, and let
\[
\Omega_g:=\{0\}\cup\{x\ne0:\ |x|<g(\theta(x))\}.
\]
Then $\Omega_g$ is a bounded admissible domain, star-shaped with respect to $0$, containing $B_{g(\pi/2)}$ and contained in $B_{g(0)}$. Conversely, if $g$ is lower semicontinuous and $\Omega_g$ is admissible, then $g$ is non-increasing.
\end{lemma}

\begin{proof}
The map $x\mapsto\theta(x)$ is continuous on $\R^N\setminus\{0\}$. Hence $x\mapsto g(\theta(x))-|x|$ is lower semicontinuous there, and $\{x\ne0:\ g(\theta(x))-|x|>0\}$ is open. Since $g\ge g(\pi/2)>0$, we have $B_{g(\pi/2)}\subset\Omega_g$, so $\Omega_g$ is open. Clearly $\Omega_g\subset B_{g(0)}$. If $x\in\Omega_g$ and $t\in(0,1)$, then $\theta(tx)=\theta(x)$ and $|tx|<|x|$, so $tx\in\Omega_g$; thus $\Omega_g$ is star-shaped, hence connected.

Since $\theta(x)$ depends only on $|x'|$ and $|x_N|$, $\Omega_g$ is $\G$-invariant. If $x\in\Omega_g\setminus\{0\}$, $|y|=|x|$ and $|y_N|\le|x_N|$, then $\theta(y)\le\theta(x)$, so $g(\theta(y))\ge g(\theta(x))>|x|=|y|$ and $y\in\Omega_g$.

Conversely, suppose $\theta_1<\theta_2$ and $g(\theta_1)<g(\theta_2)$. Pick $r\in(g(\theta_1),g(\theta_2))$, $x=r(\cos\theta_2 e_1+\sin\theta_2 e_N)\in\Omega_g$ and $y=r(\cos\theta_1e_1+\sin\theta_1e_N)$. Then $|y|=|x|$ and $|y_N|<|x_N|$, but $y\notin\Omega_g$ since $|y|=r>g(\theta_1)$. So condition (b) fails.
\end{proof}

If $g$ is continuous, $\partial\Omega_g=\{|x|=g(\theta(x))\}$. If $g$ is Lipschitz, $\Omega_g$ is a Lipschitz domain and, in particular, satisfies the exterior cone condition.  Note that $\min\{g_1,g_2\}$ and $\max\{g_1,g_2\}$ are non-increasing if $g_1,g_2$ are, so intersections and unions of admissible radial-graph domains are admissible.

 Next we describe the examples in Figure~\ref{fig:admissible}.

\begin{example}[Ellipsoids of revolution]\label{ex:ellipsoid}
Let $a,b>0$ and
\[
E_{a,b}:=\Big\{x\in\R^N:\ \frac{|x'|^2}{a^2}+\frac{x_N^2}{b^2}<1\Big\}.
\]
For $x\ne0$ with $|x|=r$ and $\theta(x)=\theta$ we have $|x'|=r\cos\theta$ and $|x_N|=r\sin\theta$, so
\[
x\in E_{a,b}\iff r^2\,c(\theta)<1,\qquad c(\theta):=\frac{\cos^2\theta}{a^2}+\frac{\sin^2\theta}{b^2}.
\]
Thus $E_{a,b}=\Omega_g$ with $g(\theta)=c(\theta)^{-1/2}$, and
\[
c'(\theta)=2\sin\theta\cos\theta\Big(\frac1{b^2}-\frac1{a^2}\Big)=\sin(2\theta)\Big(\frac1{b^2}-\frac1{a^2}\Big),\qquad g'(\theta)=-\frac12c(\theta)^{-3/2}c'(\theta).
\]
\begin{itemize}
\item If $b\le a$ (oblate case, including the ball $a=b$), then $c'\ge0$ on $[0,\pi/2]$, $g$ is non-increasing and $E_{a,b}$ is admissible by Lemma~\ref{lem:polar}.
\item If $b>a$ (prolate case), $g$ is strictly increasing and $E_{a,b}$ is not admissible. This also follows directly from Lemma~\ref{lem:disc}: $R_{E_{a,b}}=b$, but the equatorial points of $E_{a,b}$ satisfy $|x'|<a<b$.
\end{itemize}
In $\R^3$, Theorem~\ref{thm:main} therefore applies to every oblate spheroid $\{(x_1^2+x_2^2)/a^2+x_3^2/b^2<1\}$, $b\le a$. Since these domains are smooth and star-shaped with respect to $0$, the Pohozaev identity shows that the range $1<p<p_\alpha$ is optimal. Ellipsoids with different semi-axes in the $x'$ directions are not $O(N-1)$-invariant and are not covered.
\end{example}

\begin{example}[Lenses and unions of balls]\label{ex:lens}
Let $0<c<R$ and
\[
L_{R,c}:=B_R(ce_N)\cap B_R(-ce_N).
\]
$L_{R,c}$ is $\G$-invariant and $0\in L_{R,c}$. For $x_N\ge0$ we have $|x-ce_N|\le|x+ce_N|$, so for such $x$, $x\in L_{R,c}$ iff $|x+ce_N|<R$. With $|x|=r$ and $\theta(x)=\theta$, this reads
\[
r^2+2cr\sin\theta+c^2<R^2 .
\]
The left-hand side is increasing in $r\ge0$, so this is equivalent to $r<g(\theta)$ with
\[
g(\theta)=-c\sin\theta+\sqrt{R^2-c^2\cos^2\theta},
\]
the positive root of $r^2+2cr\sin\theta+c^2-R^2=0$. As a function of $s=\sin\theta\in[0,1]$, $g=G(s):=-cs+\sqrt{R^2-c^2+c^2s^2}$ and
\[
G'(s)=-c+\frac{c^2s}{\sqrt{R^2-c^2+c^2s^2}}<0,
\]
because $c^2s^2<R^2-c^2+c^2s^2$ (i.e.\ $c<R$).  Hence $g$ is strictly decreasing, from $g(0)=\sqrt{R^2-c^2}$ to $g(\pi/2)=R-c$, and $L_{R,c}=\Omega_g$ is admissible. The boundary of $L_{R,c}$ has an edge along the equatorial sphere $\{x_N=0,\ |x'|=\sqrt{R^2-c^2}\}$. So $L_{R,c}$ is a Lipschitz, non-smooth, convex admissible domain. 
By contrast, the union $B_R(ce_N)\cup B_R(-ce_N)$ corresponds to
\[
g(\theta)=c\sin\theta+\sqrt{R^2-c^2\cos^2\theta},
\]
which is strictly increasing. So the union is not admissible. 
\end{example}

\subsection{Domains bounded by two graphs}

Many axially symmetric domains are naturally described as $\{|x_N|<h(|x'|)\}$. The following proposition translates admissibility stated in Lemma \ref{lem:polar} into a condition on $h$.

\begin{proposition}\label{prop:graph}
Let $a>0$ and $h\in C([0,a])$ with $h>0$ on $[0,a)$ and $h(a)=0$. Assume that
\begin{enumerate}[label=\textup{(H\arabic*)}]
\item $s\mapsto h(s)/s$ is strictly decreasing on $(0,a]$;
\item $s\mapsto m(s):=s^2+h(s)^2$ is non-decreasing on $[0,a]$.
\end{enumerate}
Then
\[
\Omega_h:=\{x\in\R^N:\ |x'|<a,\ |x_N|<h(|x'|)\}
\]
is admissible. Conversely, if \textup{(H1)} holds and $m$ is not non-decreasing, then $\Omega_h$ is not admissible.
\end{proposition}

\begin{proof}
By (H1)  for every $\theta\in(0,\pi/2)$ there is  a unique $s(\theta)\in(0,a)$ with
\[
\frac{h(s(\theta))}{s(\theta)}=\tan\theta .
\]
The map $\theta\mapsto s(\theta)$ is continuous and strictly decreasing. We set $s(0):=a$ and $s(\pi/2):=0$, and define
\[
g(\theta):=\sqrt{m(s(\theta))}=\sqrt{s(\theta)^2+h(s(\theta))^2}.
\]
Geometrically, $g(\theta)$ is the distance from $0$ of the point $(s(\theta),h(s(\theta)))$ of the boundary curve seen at latitude $\theta$; note that $g(\theta)=s(\theta)/\cos\theta$ for $\theta<\pi/2$.

\emph{Claim: $\Omega_h=\Omega_g$.} Let $x\ne0$, $r=|x|$, $\theta=\theta(x)$, so that $|x'|=r\cos\theta$ and $|x_N|=r\sin\theta$.
\begin{itemize}
\item If $\theta=0$: $x\in\Omega_h$ iff $|x'|<a$ iff $r<a=g(0)$.
\item If $\theta=\pi/2$: $x\in\Omega_h$ iff $|x_N|<h(0)$ iff $r<g(\pi/2)$.
\item If $0<\theta<\pi/2$ and $|x'|=r\cos\theta\ge a$, then $x\notin\Omega_h$, and also $r\ge a/\cos\theta>s(\theta)/\cos\theta=g(\theta)$.
\item If $0<\theta<\pi/2$ and $|x'|<a$, then $x\in\Omega_h$ iff $\tan\theta<h(|x'|)/|x'|$ iff $|x'|<s(\theta)$ by (H1), iff $r<s(\theta)/\cos\theta=g(\theta)$.
\end{itemize}
So $\Omega_h=\Omega_g$. The function $g$ is continuous, and it is non-increasing if and only if $m$ is non-decreasing, because $\theta\mapsto s(\theta)$ is a decreasing bijection of $[0,\pi/2]$ onto $[0,a]$. Both assertions now follow from Lemma~\ref{lem:polar}. 
\end{proof}

\begin{example}[Ellipsoids, again]
For $h(s)=b\sqrt{1-s^2/a^2}$ we have $h(s)/s=b\sqrt{1/s^2-1/a^2}$, which is strictly decreasing, and
\[
m(s)=b^2+\Big(1-\frac{b^2}{a^2}\Big)s^2 ,
\]
which is non-decreasing iff $b\le a$. This recovers Example~\ref{ex:ellipsoid}.
\end{example}

\begin{example}[Double paraboloids (``flying saucers'')]\label{ex:saucer}
Let $a,b>0$ and $h(s)=b(1-s^2/a^2)$, so that
\[
\Omega_h=\Big\{|x'|<a,\ |x_N|<b\Big(1-\frac{|x'|^2}{a^2}\Big)\Big\}
\]
is bounded by two paraboloids meeting along the equatorial sphere $\{x_N=0,\ |x'|=a\}$ with a corner. (H1) holds since $h(s)/s=b(1/s-s/a^2)$ is strictly decreasing. Moreover
\[
m(s)=s^2+b^2\Big(1-\frac{s^2}{a^2}\Big)^2,\qquad m'(s)=2s\Big[1-\frac{2b^2}{a^2}\Big(1-\frac{s^2}{a^2}\Big)\Big].
\]
The bracket is increasing in $s\in[0,a]$ and equals $1-2b^2/a^2$ at $s=0$. Therefore
\[
\Omega_h\ \text{is admissible}\iff b\le\frac a{\sqrt2}.
\]
If $a/\sqrt2<b\le a$, the domain is flatter than a ball ($b\le a$) but still not admissible: $m'<0$ near $s=0$, i.e.\ near the poles the boundary is more curved than the sphere through the pole. 
\end{example}

\subsection{Non-convex and non star-shaped examples}

\begin{example}[Ball with an equatorial collar]\label{ex:collar}
Let $0<b<a$ and $\theta_0\in(0,\pi/2)$, and let
\[
g(\theta)=\begin{cases}a&0\le\theta<\theta_0,\\ b&\theta_0\le\theta\le\pi/2.\end{cases}
\]
$g$ is non-increasing and lower semicontinuous (it is right-continuous at $\theta_0$). By Lemma~\ref{lem:polar}, the domain
\[
\Omega_g=B_b\cup\big\{x\in B_a:\ |x_N|<\tan\theta_0\,|x'|\big\}
\]
is admissible. It consists of a ball together with a thick collar around the equator. It is star-shaped with respect to $0$, it has a Lipschitz boundary,  but neither convex nor smooth. More generally, every non-increasing lower semicontinuous step function $g$ gives an admissible domain made of concentric spherical sectors of decreasing radius.
\end{example}

\begin{example}[A non star-shaped admissible domain]\label{ex:nonstar}
Let $0<a<b<c$ and $\theta_0\in(0,\pi/2)$, and let
\[
\Omega:=B_c\setminus T,\qquad T:=\{x\ne0:\ a\le|x|\le b,\ \theta(x)\ge\theta_0\}.
\]
Then:
\begin{itemize}
\item $T$ is compact and $\G$-invariant, so $\Omega$ is open, bounded, $\G$-invariant, and $0\in\Omega$.
\item $\Omega$ is connected: $B_a\subset\Omega$, and every point of $\Omega$ can be joined inside $\Omega$ to $B_a$, either radially or through the equatorial zone $\{\theta<\theta_0\}$.
\item $\Omega$ is admissible. Let $x\in\Omega$ and $|y|=|x|$, $|y_N|\le|x_N|$, so that $\theta(y)\le\theta(x)$. If $|x|\notin[a,b]$, then $y\in B_c\setminus T$. If $|x|\in[a,b]$, then $x\notin T$ gives $\theta(x)<\theta_0$, hence $\theta(y)<\theta_0$ and $y\notin T$.
\item $\Omega$ is not star-shaped with respect to $0$.
\end{itemize}
\end{example}

\begin{remark}[Slice description]
An open $\G$-invariant set $\Omega$ satisfies condition (b) iff, for every $r>0$, the set $\{|y_N|/r:\ y\in\Omega\cap\partial B_r\}\subset[0,1]$ is either empty, or an interval $[0,\varrho(r))$ with $0<\varrho(r)\le1$, or $[0,1]$. Indeed it is down-closed by (b) and relatively open by openness and $\G$-invariance of $\Omega$. So $\Omega$ is admissible iff each slice $\Omega\cap\partial B_r$ is empty, the whole sphere, or an open zone $\{|y_N|<r\varrho(r)\}$ symmetric around the equator. No monotonicity of $r\mapsto\varrho(r)$ is required, which explains Example~\ref{ex:nonstar}.
\end{remark}

\subsection{Examples non-admissible}

\begin{example}[Cylinders]\label{ex:cyl}
Let $\mathcal C:=\{|x'|<a,\ |x_N|<\ell\}$ with $a,\ell>0$. Then $R=\sqrt{a^2+\ell^2}>a$, while the equatorial points of $\mathcal C$ satisfy $|x'|<a$. By Lemma~\ref{lem:disc}, $\mathcal C$ is not admissible. In polar form, $\mathcal C=\Omega_g$ with
\[
g(\theta)=\min\Big\{\frac a{\cos\theta},\frac\ell{\sin\theta}\Big\},
\]
which is strictly increasing on $[0,\arctan(\ell/a)]$. More generally, Lemma~\ref{lem:disc} excludes every $\G$-invariant domain whose equatorial section is smaller than the ball $B_{R_\Omega}$. 
\end{example}

\begin{example}[Other non-examples]
Prolate ellipsoids (Example~\ref{ex:ellipsoid}), unions of two balls centred on the axis (Example~\ref{ex:lens}), and double paraboloids with $b>a/\sqrt2$ (Example~\ref{ex:saucer}) are not admissible. The closures of the first two are admissible obstacles in the sense of Definition~\ref{def:polar} (Section~\ref{sec:kelvin}). Balls $B_R(ce_N)$ with $0<|c|<R$ contain $0$ and are axially symmetric, but they are not symmetric with respect to $\{x_N=0\}$, so condition (a) fails.
\end{example}

%=====================================================================
\section{The Kelvin transform: exterior domains and further examples}\label{sec:kelvin}
%=====================================================================
In this section we apply the Kelvin transform with respect to the origin to \eqref{eq:P}. It maps \eqref{eq:P} from  a bounded domain containing $0$ to a H\'enon-type problem in an exterior domain, and positive solutions which are regular at the origin to solutions with fast decay at infinity. 
For a function $u$ defined on $\Omega\setminus\{0\}$, its \emph{Kelvin transform} is
\[
v(y):=|y|^{2-N}u(y/|y|^2),\qquad y\in\Omega^*
\]
where for $A\subset\R^N$ we write $A^*:=\{x/|x|^2:\ x\in A\setminus\{0\}\}$.
We have the following: 

\begin{lemma}\label{lem:kelvin}
Let $\Omega$ be a bounded open set with $0\in\Omega$, $\alpha\in\R$, $p>1$, and let $v$ be the Kelvin transform of $u\in C^2(\Omega\setminus\{0\})$. Then:
\begin{enumerate}[label=\textup{(\roman*)}]
\item if $-\Delta u=|x|^\alpha u^p$ in $\Omega\setminus\{0\}$, then
\begin{equation}\label{eq:Pstar}
-\Delta v=|y|^\beta v^p\quad\text{in }\Omega^*,\qquad \beta:=(N-2)p-(N+2+\alpha) ;
\end{equation}
\item if $u\ge0$, then $\displaystyle\int_{\Omega^*}|y|^\beta v^{p+1}\dd y=\int_\Omega|x|^\alpha u^{p+1}\dd x$.
%\item for $p_\beta:=\frac{N+2+2\beta}{N-2}$, we have $p<p_\alpha\iff\beta<\alpha\iff p>p_\beta$ and $p=p_\alpha\iff p=p_\beta$. {\Ma Quest'ultima riga secondo me si può levare alla luce della dimostrazione del Theorem \ref{thm:ext} che ti ho proposto.}
\end{enumerate}
\end{lemma}
 We can give the following definition:

\begin{definition}\label{def:polar}
A compact set $D\subset\R^N$ is an \emph{admissible obstacle} if
\begin{enumerate}[label=\textup{(\alph*$^*$)}]
\item $0$ is an interior point of $D$ and $\R^N\setminus D$ is connected;
\item $D$ is $\G$-invariant;
\item for every $y\in D$ and every $z\in\R^N$ with $|z|=|y|$ and $|z_N|\ge|y_N|$, we have $z\in D$.
\end{enumerate}
\end{definition}
Condition (c$^*$) is the counterpart of condition (b) of Definition~\ref{def:adm}: moving a point of $D$ along a meridian towards the poles never leaves $D$. 

\begin{lemma}\label{lem:dualdom}
Let $D\subset\R^N$ be compact with $0\in\mathrm{int}\,D$, and let
\[
\Omega_D:=\{0\}\cup(\R^N\setminus D)^*.
\]
Then $\Omega_D$ is a bounded open set containing $0$, and $D$ is an admissible obstacle if and only if $\Omega_D$ is an admissible domain. Conversely, if $\Omega$ is an admissible domain, then $D_\Omega:=\{0\}\cup(\R^N\setminus\Omega)^*$ is an admissible obstacle and $\Omega_{D_\Omega}=\Omega$.
\end{lemma}
Condition (b) for $\Omega_D$ and condition (c$^*$) for $D$ are equivalent because the inversion $x\mapsto x/|x|^2$ maps $\partial B_r$ onto $\partial B_{1/r}$ and preserves the latitude.

We can now prove Theorem~\ref{thm:ext} and Corollary \ref{cor:ext}.
\begin{proof}[Proof of Theorem \ref{thm:ext}.]
Let $p$ be as in the statement and set $\alpha:=(N-2)p-(N+2+\beta)$. Then $\alpha>0$ because $p>\frac{N+2+\beta}{N-2}$. Moreover $p<p_\alpha$ is equivalent to $(N-2)p<N+2+2\alpha=2(N-2)p-(N+2)-2\beta$, i.e.\ to $p>p_\beta$. By Lemma~\ref{lem:dualdom}, $\Omega_D$ is admissible, and $\partial\Omega_D=(\partial D)^*$ is Lipschitz, since the inversion is a diffeomorphism near the compact set $\partial D\not\ni0$. Hence by Theorem~\ref{thm:main} there exists a solution $u\in C(\overline{\Omega_D})$ of \eqref{eq:P} in $\Omega_D$. Let $v$ be its Kelvin transform, defined on $(\Omega_D\setminus\{0\})^*=\R^N\setminus D$ which satisfies $v>0$ in $\R^N\setminus D$ and $v=0$ on $\partial D$ because $u=0$ on $\partial\Omega_D=\partial D^*$.
Then (i) follows from $|y|^{N-2}v(y)=u(y/|y|^2)\to u(0)>0$; (ii) from Lemma~\ref{lem:kelvin}(ii) and, for the gradient, from $|\nabla v(y)|\le C|y|^{1-N}$ for $|y|$ large, since $u\in C^1$ near $0$; (iii) from Theorem~\ref{thm:main}, since $y\mapsto y/|y|^2$ commutes with $\G$, maps $\partial B_r$ onto $\partial B_{1/r}$ and preserves the latitude.
\end{proof}
\begin{proof}[Proof of Corollary \ref{cor:ext}.]
 Take $\beta=0$ in Theorem~\ref{thm:ext}: then $p_0=\frac{N+2}{N-2}$.  Both conditions on $p$ reduce to $p>\frac{N+2}{N-2}$ and the claim follows.
\end{proof}

 We give here some examples of admissible obstacles and the transformed domains.

\begin{example}[Lenses and pairs of balls]\label{ex:lensdual}
Let $0<c<R$ and let $L_{R,c}=B_R(ce_N)\cap B_R(-ce_N)$ be the lens of Example~\ref{ex:lens}. Then
\[
\{0\}\cup(\R^N\setminus L_{R,c})^*=\overline{B_\rho(de_N)}\cup\overline{B_\rho(-de_N)},\qquad d=\frac{c}{R^2-c^2},\quad \rho=\frac{R}{R^2-c^2}.
\]
Since $(R,c)\mapsto(\rho,d)$ is a bijection of $\{0<c<R\}$ onto $\{0<d<\rho\}$, the pairs of overlapping balls $\overline{B_\rho(\pm de_N)}$, $0<d<\rho$, are exactly the admissible obstacles dual to lenses (Figure~\ref{fig:duality}).
\end{example}
\begin{example}[Unions of balls centred on the axis]\label{ex:beads2}
Let $A\subset[0,\infty)\times(0,\infty)$ be compact and $D_A:=\bigcup_{(t,\rho)\in A}\big(\overline{B_\rho(te_N)}\cup\overline{B_\rho(-te_N)}\big)$. Then $D_A$ is compact and $\G$-invariant. Since $\min_\pm|y\mp te_N|^2=|y|^2-2t|y_N|+t^2$ is non-increasing in $|y_N|$ for fixed $|y|$, $D_A$ satisfies (c$^*$). Moreover $\R^N\setminus D_A$ is open and connected: for $y\notin D_A$ with $y'\neq0$, the path $\lambda\mapsto((1+\lambda)y',y_N)$, $\lambda\ge0$, goes to infinity without meeting $D_A$. Hence $D_A$ is an admissible obstacle if and only if $0\in\mathrm{int}\,D_A$, and Corollary~\ref{cor:ext} applies. This includes \emph{capsules}, $A=[0,\ell]\times\{\rho\}$, and \emph{strings of beads}, $A=\{(0,\rho_0),(t_1,\rho_1),\dots,(t_k,\rho_k)\}$: a central ball with $k$ symmetric pairs of balls on the axis, overlapping or not (Figure~\ref{fig:duality}, bottom).
\end{example}
\begin{example}[Quartic domains and obstacles]\label{ex:quartic}
With $c(\theta)$ as in Example~\ref{ex:ellipsoid}, $E_{a,b}=\{|y|<c(\theta)^{-1/2}\}$, and its dual is the quartic set
\[
Q_{a,b}:=\{0\}\cup\Big\{x:\ |x|^4<\frac{|x'|^2}{a^2}+\frac{x_N^2}{b^2}\Big\}.
\]
If $b\ge a$, $\overline{E_{a,b}}$ is an admissible obstacle and $Q_{a,b}$ an admissible domain; if $b\le a$, $E_{a,b}$ is admissible and $\overline{Q_{a,b}}$ is an admissible obstacle.
\end{example}
\definecolor{admF}{HTML}{BCD3EE}% admissible domain: fill
\definecolor{admE}{HTML}{2B6CB0}% admissible domain: boundary and arcs
\definecolor{obsF}{HTML}{F6CBA6}% obstacle: fill
\definecolor{obsE}{HTML}{C0562A}% obstacle: boundary and arcs
\definecolor{extF}{HTML}{EEF3FA}% exterior domain R^N \ D: shading

\begin{center}\begin{minipage}{\textwidth}\centering
\begin{tikzpicture}[
  x=0.1155\textwidth,
  y=0.1155\textwidth,
  panel/.style={draw=black,line width=.6pt},
  axes/.style={gray!70,dashed,line width=.35pt,dash pattern=on 2pt off 1.5pt},
  guide/.style={gray!80!black,densely dotted,line width=.45pt},
  admedge/.style={draw=admE,line width=.8pt},
  obsedge/.style={draw=obsE,line width=.8pt},
  admarc/.style={draw=admE,line width=1.7pt,line cap=butt},
  obsarc/.style={draw=obsE,line width=1.7pt,line cap=butt},
  title/.style={font=\footnotesize,anchor=base}
]

% Panels are squares [-1,1]^2; columns at x = 0, 2.2, 4.4, 6.6;
% rows at y = 0, -2.45.

\def\frameA{\draw[panel] (-1,-1) rectangle (1,1);}
\def\frameO{\fill[extF] (-1,-1) rectangle (1,1);}
\def\axes{\draw[axes] (-1,0)--(1,0) (0,-1)--(0,1);}
\def\origin{\fill (0,0) circle[radius=.022];}

% =================================================================
% 1--2. LENS AND ITS DUAL
% =================================================================

\pgfmathsetmacro{\lR}{1.3}
\pgfmathsetmacro{\lc}{0.7}
\pgfmathsetmacro{\lr}{0.8}

\pgfmathsetmacro{\lw}{sqrt(\lR*\lR-\lc*\lc)}
\pgfmathsetmacro{\lal}{atan(\lc/\lw)}
\pgfmathsetmacro{\lth}{
  asin((\lR*\lR-\lc*\lc-\lr*\lr)/(2*\lc*\lr))
}
\pgfmathsetmacro{\ld}{\lc/(\lR*\lR-\lc*\lc)}
\pgfmathsetmacro{\lrho}{\lR/(\lR*\lR-\lc*\lc)}

% 1. lens ----------------------------------------------------------

\begin{scope}[shift={(0,0)}]

  \node[title] at (0,1.12)
    {$\Om_D=$ lens (Ex.~\ref{ex:lensdual})};

  \axes

  \begin{scope}[scale=.80]

    \def\dom{
      (\lw,0)
      arc[start angle=\lal,end angle=180-\lal,radius=\lR]
      arc[start angle=180+\lal,end angle=360-\lal,radius=\lR]
      -- cycle
    }

    \fill[admF] \dom;

    \draw[guide] (0,0) circle[radius=\lr];

    \begin{scope}
      \clip \dom;

      \draw[admarc]
        (-\lth:\lr)
        arc[start angle=-\lth,end angle=\lth,radius=\lr];

      \draw[admarc]
        (180-\lth:\lr)
        arc[start angle=180-\lth,end angle=180+\lth,radius=\lr];

    \end{scope}

    \draw[admedge] \dom;

  \end{scope}

  \origin
  \frameA

\end{scope}

% 2. dual of the lens: two balls ----------------------------------

\begin{scope}[shift={(2.2,0)}]

  \node[title] at (0,1.12)
    {$D=$ two balls};

  \frameO
  \axes

  \begin{scope}[scale=.55]

    \def\dom{
      (0,\ld) circle[radius=\lrho]
      (0,-\ld) circle[radius=\lrho]
    }

    \draw[obsE,line width=1.6pt] \dom;
    \fill[obsF] \dom;

    \pgfmathsetmacro{\gr}{1/\lr}

    \draw[guide] (0,0) circle[radius=\gr];

    \begin{scope}
      \clip \dom;

      \draw[obsarc]
        (\lth:\gr)
        arc[start angle=\lth,end angle=180-\lth,radius=\gr];

      \draw[obsarc]
        (180+\lth:\gr)
        arc[start angle=180+\lth,end angle=360-\lth,radius=\gr];

    \end{scope}

  \end{scope}

  \origin
  \frameA

\end{scope}

% =================================================================
% 3--4. STRING OF BEADS AND ITS DUAL
% moved from the old positions 7--8
% =================================================================

\def\beads{0.7/0.42,1.2/0.32,1.6/0.22}

\pgfmathsetmacro{\bz}{0.55}
\pgfmathsetmacro{\bt}{1.25}

\pgfmathsetmacro{\bth}{
  asin((\bt*\bt+1.2*1.2-0.32*0.32)/(2*\bt*1.2))
}

% 3. dual of a string of beads ------------------------------------

\begin{scope}[shift={(4.4,0)}]

  \node[title] at (0,1.12)
    {$\Omega_{D_A}$ (Ex. \ref{ex:beads2})};

  % overlay: the large removed balls must not enlarge the bounding box
  \begin{scope}[scale=.50,overlay]

    \pgfmathsetmacro{\BR}{1/\bz}
    \pgfmathsetmacro{\gr}{1/\bt}

    \fill[admF] (0,0) circle[radius=\BR];
    \draw[admedge] (0,0) circle[radius=\BR];

    % removed balls: boundary inside the big ball
    \begin{scope}
      \clip (0,0) circle[radius=\BR];

      \foreach \tj/\rj in \beads {

        \pgfmathsetmacro{\mj}{
          \tj/(\tj*\tj-\rj*\rj)
        }

        \pgfmathsetmacro{\sj}{
          \rj/(\tj*\tj-\rj*\rj)
        }

        \draw[admE,line width=1.6pt]
          (0,\mj) circle[radius=\sj]
          (0,-\mj) circle[radius=\sj];
      }

    \end{scope}

    % white holes, clipped to the panel
    \clip (-2,-2) rectangle (2,2);

    \foreach \tj/\rj in \beads {

      \pgfmathsetmacro{\mj}{
        \tj/(\tj*\tj-\rj*\rj)
      }

      \pgfmathsetmacro{\sj}{
        \rj/(\tj*\tj-\rj*\rj)
      }

      \fill[white]
        (0,\mj) circle[radius=\sj]
        (0,-\mj) circle[radius=\sj];

    }

  \end{scope}

  \axes

  \begin{scope}[scale=.50]

    \pgfmathsetmacro{\gr}{1/\bt}

    \draw[guide] (0,0) circle[radius=\gr];

    \draw[admarc]
      (-\bth:\gr)
      arc[start angle=-\bth,end angle=\bth,radius=\gr];

    \draw[admarc]
      (180-\bth:\gr)
      arc[start angle=180-\bth,end angle=180+\bth,radius=\gr];

  \end{scope}

  \origin
  \frameA

\end{scope}

% 4. string of beads D_A -------------------------------------------

\begin{scope}[shift={(6.6,0)}]

  \node[title] at (0,1.12)
    { $D_A=$ string of beads};

  \frameO
  \axes

  \begin{scope}[scale=.50]

    \def\dom{
      (0,0) circle[radius=\bz]
      (0,0.7) circle[radius=0.42]
      (0,-0.7) circle[radius=0.42]
      (0,1.2) circle[radius=0.32]
      (0,-1.2) circle[radius=0.32]
      (0,1.6) circle[radius=0.22]
      (0,-1.6) circle[radius=0.22]
    }

    \draw[obsE,line width=1.6pt] \dom;
    \fill[obsF] \dom;

    \draw[guide] (0,0) circle[radius=\bt];

    \begin{scope}
      \clip \dom;

      \draw[obsarc]
        (\bth:\bt)
        arc[start angle=\bth,end angle=180-\bth,radius=\bt];

      \draw[obsarc]
        (180+\bth:\bt)
        arc[start angle=180+\bth,end angle=360-\bth,radius=\bt];

    \end{scope}

  \end{scope}

  \origin
  \frameA

\end{scope}

% =================================================================
% 5--6. QUARTIC DOMAIN AND PROLATE ELLIPSOID
% =================================================================

\pgfmathsetmacro{\qa}{1}
\pgfmathsetmacro{\qb}{2}
\pgfmathsetmacro{\qr}{0.8}

\pgfmathsetmacro{\qth}{
  asin(
    sqrt(
      (1/(\qa*\qa)-\qr*\qr)/
      (1/(\qa*\qa)-1/(\qb*\qb))
    )
  )
}

% 5. quartic Q_{1,2} -----------------------------------------------

\begin{scope}[shift={(0,-2.45)}]

  \node[title] at (0,1.12)
    {$\Om_D=Q_{1,2}$ (Ex. \ref{ex:quartic})};

  \axes

  \begin{scope}[scale=.90]

    \def\dom{
      plot[
        domain=0:360,
        samples=181,
        smooth
      ]
      (\x:{
        sqrt(
          cos(\x)*cos(\x)/(\qa*\qa)
          +sin(\x)*sin(\x)/(\qb*\qb)
        )
      })
      -- cycle
    }

    \fill[admF] \dom;

    \draw[guide] (0,0) circle[radius=\qr];

    \begin{scope}
      \clip \dom;

      \draw[admarc]
        (-\qth:\qr)
        arc[start angle=-\qth,end angle=\qth,radius=\qr];

      \draw[admarc]
        (180-\qth:\qr)
        arc[start angle=180-\qth,end angle=180+\qth,radius=\qr];

    \end{scope}

    \draw[admedge] \dom;

  \end{scope}

  \origin
  \frameA

\end{scope}

% 6. dual: prolate ellipsoid E_{1,2} -------------------------------

\begin{scope}[shift={(2.2,-2.45)}]

  \node[title] at (0,1.12)
{$D=$ prolate ellipsoid};
  \frameO
  \axes

  \begin{scope}[scale=.46]

    \pgfmathsetmacro{\gr}{1/\qr}

    \def\dom{
      (0,0) ellipse[
        x radius=\qa,
        y radius=\qb
      ]
    }

    \fill[obsF] \dom;

    \draw[guide] (0,0) circle[radius=\gr];

    \begin{scope}
      \clip \dom;

      \draw[obsarc]
        (\qth:\gr)
        arc[start angle=\qth,end angle=180-\qth,radius=\gr];

      \draw[obsarc]
        (180+\qth:\gr)
        arc[start angle=180+\qth,end angle=360-\qth,radius=\gr];

    \end{scope}

    \draw[obsedge] \dom;

  \end{scope}

  \origin
  \frameA

\end{scope}

% =================================================================
% 7--8. BALL WITH HOLES AND ITS DUAL
% moved from the old positions 3--4
% =================================================================

\pgfmathsetmacro{\hR}{1}
\pgfmathsetmacro{\hm}{0.55}
\pgfmathsetmacro{\hs}{0.20}

\pgfmathsetmacro{\hr}{0.55}

\pgfmathsetmacro{\hth}{
  asin(
    (\hr*\hr+\hm*\hm-\hs*\hs)/
    (2*\hr*\hm)
  )
}

\pgfmathsetmacro{\hmm}{
  \hm/(\hm*\hm-\hs*\hs)
}

\pgfmathsetmacro{\hss}{
  \hs/(\hm*\hm-\hs*\hs)
}

% 7. ball with holes on the axis -----------------------------------

\begin{scope}[shift={(4.4,-2.45)}]

  \node[title] at (0,1.12)
    {$\Om_D=B_R$ with holes (Ex. \ref{ex:holes})};

  \begin{scope}[scale=.95]

    \def\holes{
      (0,\hm) circle[radius=\hs]
      (0,-\hm) circle[radius=\hs]
    }

    \fill[admF] (0,0) circle[radius=\hR];
    \draw[admedge] (0,0) circle[radius=\hR];

    \draw[admE,line width=1.6pt] \holes;
    \fill[white] \holes;

  \end{scope}

  \axes

  \begin{scope}[scale=.95]

    \draw[guide] (0,0) circle[radius=\hr];

    \draw[admarc]
      (-\hth:\hr)
      arc[start angle=-\hth,end angle=\hth,radius=\hr];

    \draw[admarc]
      (180-\hth:\hr)
      arc[start angle=180-\hth,end angle=180+\hth,radius=\hr];

  \end{scope}

  \origin
  \frameA

\end{scope}

% 8. dual: three disjoint balls ------------------------------------

\begin{scope}[shift={(6.6,-2.45)}]

  \node[title] at (0,1.12)
    {$D$= three disjoint balls};

  \frameO
  \axes

  \begin{scope}[scale=.32]

    \pgfmathsetmacro{\cR}{1/\hR}
    \pgfmathsetmacro{\gr}{1/\hr}

    \def\dom{
      (0,0) circle[radius=\cR]
      (0,\hmm) circle[radius=\hss]
      (0,-\hmm) circle[radius=\hss]
    }

    \fill[obsF] \dom;
    \draw[obsedge] \dom;

    \draw[guide] (0,0) circle[radius=\gr];

    \begin{scope}
      \clip \dom;

      \draw[obsarc]
        (\hth:\gr)
        arc[start angle=\hth,end angle=180-\hth,radius=\gr];

      \draw[obsarc]
        (180+\hth:\gr)
        arc[start angle=180+\hth,end angle=360-\hth,radius=\gr];

    \end{scope}

  \end{scope}

  \origin
  \frameA

\end{scope}

\end{tikzpicture}


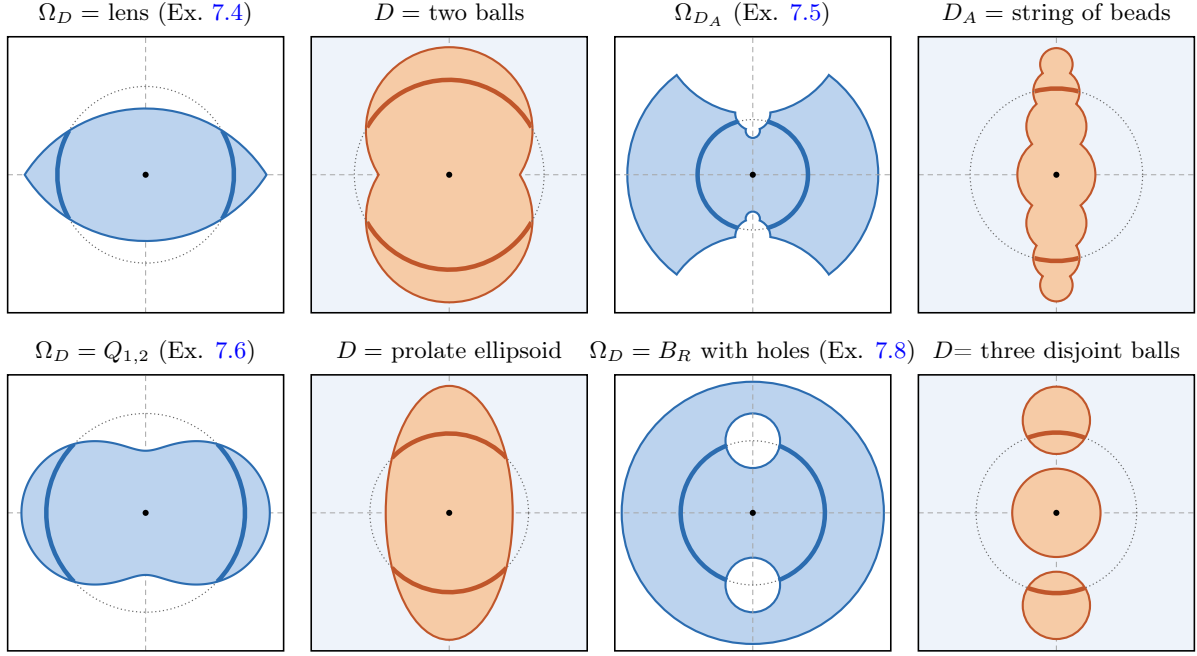
\captionof{figure}[Kelvin duality between admissible domains and admissible obstacles]{
In the figure above, the sets shown in blue are the admissible domains $\Om_D$. Next to them, in the orange figure, we show the corresponding admissible obstacles $D$, while the complements $\R^N\setminus D$ are shaded.}
\label{fig:duality}
\end{minipage}\end{center}

The notion of admissible obstacle allows one to obtain admissible domains with holes.

\begin{example}[Holes]\label{lem:holes}
Let $\Omega$ be admissible and let $H\subset\Omega$ be a compact $\G$-invariant set with $0\notin H$, satisfying \textup{(c$^*$)}, such that $\Omega\setminus H$ is connected. Then $\Omega\setminus H$ is admissible. Indeed, if $x\in\Omega\setminus H$, $|z|=|x|$ and $|z_N|\le|x_N|$, then $z\in\Omega$ by (b); if $z\in H$, then (c$^*$) gives $x\in H$, a contradiction.
\end{example}

\begin{example}[Domains with holes on the axis]\label{ex:holes}
Let $0<s<m$ with $m+s<R$, and
\[
\Omega:=B_R\setminus\Big(\overline{B_s(me_N)}\cup\overline{B_s(-me_N)}\Big).
\]
The set $H=\overline{B_s(me_N)}\cup\overline{B_s(-me_N)}$ satisfies (c$^*$) and $0\notin H$ since $s<m$. The two closed balls are disjoint and contained in $B_R$, so $\Omega$ is connected. By Example~\ref{lem:holes}, $\Omega$ is admissible. It is a smooth domain which is not star-shaped with respect to $0$, and it is not homeomorphic to a ball.  The same holds with finitely many symmetric pairs of holes centred on the axis.
The dual obstacle is
\[
\overline{B_{1/R}}\cup\overline{B_{s'}(m'e_N)}\cup\overline{B_{s'}(-m'e_N)},\qquad m'=\frac m{m^2-s^2},\quad s'=\frac s{m^2-s^2}.
\]
It consists of three disjoint closed balls since $m'-s'=\frac1{m+s}>\frac1R$.  See Figure \ref{fig:duality}.
\end{example}

%=====================================================================
%\section{Further remarks and open problems}\label{sec:final}
%=====================================================================

\subsection*{Declaration on the use of AI tools}
AI tools, specifically Claude (Anthropic),  were used to assist with the
organization and wording of the manuscript, notation checks, and preparation of the figures. The authors are responsible for the mathematical content
and references.

\subsection*{Acknowledgments}
The authors are members of the Gruppo Nazionale per l'Analisi Matematica, la Probabilit\`a e le loro Applicazioni (GNAMPA) of the Istituto Nazionale di Alta Matematica (INdAM), and were partially supported by INdAM--GNAMPA.

\end{document}